\documentclass[10pt,tgrind,psbox]{article}

\usepackage{latexsym}
\usepackage{graphics}
\usepackage{amsmath}
\usepackage{xspace}
\usepackage{amssymb}
\usepackage{psfrag}
\usepackage{epsfig}
\usepackage{amsmath,amsthm}
\usepackage{epsfig}
\usepackage{pst-all}

\newcommand {\bel}[1]{\begin{align*}}
\newcommand {\eel}[1]{\end{align*}}
\newcommand {\bea}{\begin{eqnarray}}
\newcommand {\eea}{\end{eqnarray}}

\newcommand{\pr}{\mathbb{P}}
\newcommand{\E}{\mathbb{E}}
\newcommand{\R}{\mathbb{R}}
\newcommand{\Z}{\mathbb{Z}}
\newcommand{\G}{\mathbb{G}}
\newcommand{\T}{\mathbb{T}}

\newcommand{\M}{{\bf \mathcal{M}}}
\newcommand{\F}{\mathcal{\mathcal{F}}}

\newcommand{\mb}[1]{\mbox{\boldmath $#1$}}

\newcommand{\mbI}{\mb{I}}

\newcommand{\mbM}{\mb{M}}

\newcommand{\ignore}[1]{\relax}

\newcommand{\even}{\text{even}}
\newcommand{\odd}{\text{odd}}

\newtheorem{theorem}{Theorem}
\newcommand{\remark}{{\bf Remark : }}
\newtheorem{lemma}{Lemma}

\newtheorem{prop}{Proposition}
\newtheorem{coro}{Corollary}

\newtheorem{Defi}{Definition}
\newtheorem{Assumption}{Assumption}

\title{Sequential cavity method for computing free energy and surface pressure}

\author{
 {\sf David Gamarnik }
  \thanks{Operations Research Center, MIT, Cambridge, MA,  02139, e-mail: {\tt
gamarnik@mit.edu}}
\and
{\sf Dmitriy Katz} \thanks{Operations Research Center, MIT, Cambridge, MA,  02139, e-mail: {\tt dimdim@mit.edu}}
}

\begin{document}

\maketitle

\begin{abstract}
We propose a new method for the problems of computing free energy and surface pressure
for various statistical mechanics models on a lattice $\Z^d$. Our method is based on
representing the free energy and surface pressure in terms of certain marginal probabilities in a suitably modified
sublattice of $\Z^d$. Then recent deterministic algorithms for computing marginal probabilities are used
to obtain numerical estimates of the quantities of interest.
The method works under the assumption of Strong Spatial Mixing (SSP), which
is a form of a correlation decay.

We illustrate our method for the hard-core and monomer-dimer models, and
improve several earlier estimates. For example we show that the exponent of the
monomer-dimer coverings of $\Z^3$  belongs
to the interval $[0.78595,0.78599]$, improving  best previously known
estimate of (approximately)
$[0.7850,0.7862]$ obtained in~\cite{FriedlandPeled},\cite{FriedlandKropLundowMarkstrom}.
Moreover, we show that given a target additive error $\epsilon>0$, the computational
effort of our method for these two models is $(1/\epsilon)^{O(1)}$ \emph{both} for free energy and surface pressure.
In contrast, prior methods, such as transfer matrix method, require $\exp\big((1/\epsilon)^{O(1)}\big)$ computation effort.
\end{abstract}

\tableofcontents

\section{Introduction}
It is a classical fact in statistical physics that the logarithm of the partition function of a
general statistical mechanics model on
$[-n,n]^d\subset \Z^d$, appropriately rescaled,
has a well-defined limit as $n\rightarrow\infty$~\cite{GeorgyGibbsMeasure},\cite{SimonLatticeGases}.
This limit is called free energy or pressure (the difference is in normalizing constant).
Similarly, the limit of its first order correction,  \emph{surface pressure},  exists,
though such limit depends on the shape of the underlying finite box and the existing proofs
assume soft core interactions, $H<\infty$ (see Section~\ref{section:model})~\cite{SimonLatticeGases}.
It is a different matter to compute these limits, and this question interested researchers both in
the statistical physics and combinatorics communities.
In very special cases, free energy can be computed analytically.
The most widely known example is Fisher-Kasteleyn-Temperley's formula for  dimer model on $\Z^2$ \cite{Fisher},
\cite{KasteleynDimer},\cite{TemperleyFisher}. This formula was used in~\cite{KenyonOkounkovSheffield} for
obtaining complete parametrization of different Gibbs measures for the pure dimer model on $\Z^2$. The
formula, unfortunately, extends neither to a dimer model in other dimensions nor to the closely related
monomer-dimer model. Another example of an exactly solvable model is hard-core model on a hexagonal
lattice \cite{BaxterHexagon}.

Short of these special cases, the existing methods for computing free energy mostly rely on numerical approximations.
These include randomized methods
such as Monte-Carlo~\cite{JerrumSinclairHochbaumApproxAlgorithms},\cite{KenyonRandalWilsonPermanentPlanar}
and deterministic methods such as transfer matrices.
The Monte-Carlo method can be used to estimate free energy in finite graphs, say $[-n,n]^d$,
with some probabilistic approximation
guarantee, provided that some underlying Markov chain is rapidly mixing. One then has to relate finite graph to
infinite lattices to approximate free energy for an infinite lattice. This can be done since the rate of convergence
in the definition of the free energy is known to be $O(1/n)$ different from the log-partition function on
 $[-n,n]^d$, and the constant
in $O(\cdot)$ can be explicitly bounded. The drawback of this method is its dependence on the sampling error.
Transfer matrix method, on the other hand is a deterministic method and provides rigorous bounds on the free energy.
It is based on considering an infinite strip $[-n,n]^{d-1}\times \Z$ and then identifying the number of different ways
two configurations on $[-n,n]^{d-1}$ can match. One then constructs an  $[-n,n]^{d-1}$ by $[-n,n]^{d-1}$ matrix
and the spectral radius of this transfer matrix can be related to the growth rate of the partition function
on $[-n,n]^{d-1}\times [-N,N]$ as a function of $N$, namely the partition function on $[-n,n]^{d-1}\times \Z$.
Then by making $n$ sufficiently large, bounds on the free energy on the entire lattice $\Z^d$ can be obtained.
The construction of such transfer matrix  requires time $\exp(O(n^{d-1}))$ time. Since the convergence
rate of the free energy to its limit
wrt $n$ is $O(1/n)$, then in order to get a target additive error $\epsilon$, the transfer matrix method
requires time $\exp(O((1/\epsilon)^{d-1})$. Namely, this is time required to construct an $\epsilon$-length
interval containing the actual value of the free energy.
While, for low dimensions $d$,
 this is a substantial saving over a brute force method, which would simply
compute free energy on $[-n,n]^d$ for $n\ge 1/\epsilon$
and requires time $\exp(O((1/\epsilon)^{d})$, the method stops being effective
for larger $d$. Some additional computational savings can be achieved using underlying automorphisms
group structure, see Friedland and Peled~\cite{FriedlandPeled}, but we are not aware of any formal analysis of the computational
savings produced by this method. The complexity of computing surface pressure using the transfer matrix method further
increases to $\exp(O((1/\epsilon)^{2d-2})$, see Subsection~\ref{subsection:complexity}.

The transfer matrix method was used to obtain some of the best known bounds. For the hard-core model
(see Section~\ref{section:IS}) with activity $\lambda=1$, which corresponds to counting independent sets in $\Z^d$,
the exponent of the free energy is known to be in the range $[1.503047782,1.5035148]$, as
obtained by Calkin and Wilf~\cite{CalkinWilf}.
A far more accurate but non-rigorous estimate was obtained by Baxter~\cite{Baxter99}. Similarly, the
transfer matrix method was used to compute the free energy of the monomer-dimer (matchings)
model (see Section~\ref{section:matchings}).
The problem has a long history. Earlier studies include  Hammersley \cite{HammersleyMonteCarlo},
\cite{HammersleyLowerBound}, Hammersley and Menon~\cite{HammersleyMenon}, Baxter~\cite{Baxter68},
where some non-rigorous estimates and crude bounds were obtained. Recently
Friedland and Peled~\cite{FriedlandPeled} obtained rigorously a  range $0.6627989727\pm 0.0000000001$ for $d=2$
and $[0.7653,0.7863]$ for $d=3$ using the transfer matrix method. The lower bound was later tightened to
$0.7845$ using the Friedland-Tveberg inequality~\cite{FriedlandGurvits}.  This inequality
provides a bound for general regular graphs. This bound, while quite accurate for the case $\Z^3$,
is not improvable by running some numerical procedure longer or on a faster machine. A tighter
lower bound $0.7849602275$ was recently obtained using techniques related to the asymptotic matching
conjectures~\cite{FriedlandKropLundowMarkstrom}.
A bit earlier a similar
non-rigorous estimate $[0.7833,0.7861]$ was obtained in \cite{HuoLiangLiuBai}, by reduction from a permanent problem.
Exact values for monomer-dimer entropy on two dimensional infinite strip were obtained by Kong~\cite{kong}.
We are not aware of any computational estimates of surface pressure for these or of any related statistical
mechanics models.

In this paper we propose a completely new approach for the problem of computing numerically free energy and surface pressure.
Our approach takes advantage of the fact that some of these models, including the two models above, are in the so-called
uniqueness regime. Namely, the Gibbs measure on the infinite lattice is unique. This is an implication of
the Strong Spatial Mixing (SSM) property~\cite{BertoinMartinelliPeres}, see Subsection~\ref{subsection:SSM} for the definition.
This property asserts that the marginal probability
that a node $v$ attains a particular spin value $\sigma_v$
is asymptotically independent from the spin configurations for nodes $u$ which are far away from $v$. Our main theoretical
result is Theorem~\ref{theorem:MainResultGeneral} which provides a surprisingly simple representation of the free
energy in terms of such a marginal probability, provided that the SSM holds. This representation is particularly
easy to explain for the special case of hard-core model in two dimensions. Consider the subset $\Z^2_{\prec 0}$
of $\Z^2$ consisting of points $(v_1,v_2)$ with either $v_2<0$ or $v_2=0, v_1\le 0$. Our representation
theorem states that the free energy equals $\log(1/p^*)$ where $p^*$ is the probability that a random independent
set in $\Z^2_{\prec 0}$ contains the origin $(0,0)$. As the model satisfies SSM (see Section~\ref{section:IS})
this probability is well-defined. The idea of the proof is simple. Let $v_1\prec v_2\prec\cdots\prec v_{(2n+1)^2}$
be the lexicographic ordering of nodes in $[-n,n]^2$. That is $(v_1,v_2)\succ (u_1,u_2)$ iff $v_2>u_2$ or $v_2=u_2$ and
$v_1>u_1$.
Observe the following telescoping identity $Z^{-1}=\prod_{i=0}^{(2n+1)^2-1}Z(i)/Z(i+1)$, where
$Z(i)$ is the number of independents sets in the subgraph $\G_i$ of $[-n,n]^2$ when all nodes $v_j, j>i$ and incident
edges are removed, see Figure~\ref{fig:CavityIS}. The convention $Z(0)=1$ is used.
Observe that $Z(i)/Z(i+1)$ is the probability that a randomly chosen
independent set in $\G_i$ does not contain $v_{i+1}$. This probability is then approximately $p^*$ for "most" of
the nodes in $[-n,n]^2$. The required representation is then obtained by taking logarithms of both sides. Such representation
is called cavity method in statistical physics and has been used heavily for analyzing statistical models
on sparse random graphs, \cite{Aldous:assignment00},\cite{MezardIndSets2004},
\cite{MezardParisiCavity},\cite{MezardParisi}, where one can write certain recursive
distributional equations satisfied by cavity values. Here due to a particular sequencing of removed vertices,
we call our approach \emph{sequential cavity method}.

\begin{figure}
\begin{center}
\includegraphics[scale=.3]{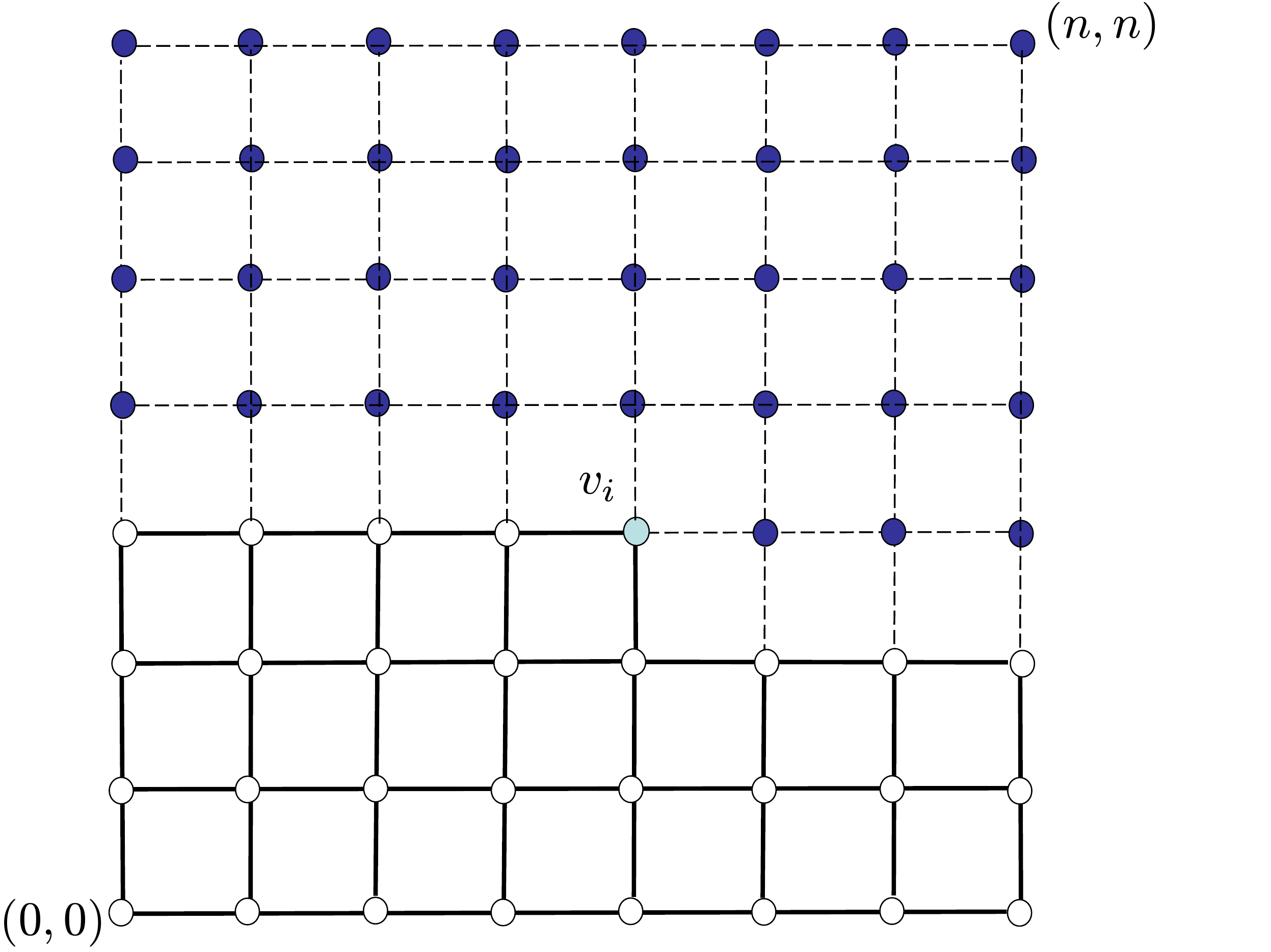}
\caption{$\G_i$ is obtained by removing all dark nodes}\label{fig:CavityIS}
\end{center}
\end{figure}

We establish a similar representation result for surface pressure for rectangular shapes again under the
assumption of (exponential) SSM. We note that our results provide an independent proof of the existence of the free energy
and surface pressure, not relying on the sub-additivity arguments, albeit in the special case of SSM.
Theorem~\ref{theorem:MainResultGeneral} reduces the problem of computing free energy and surface pressure
to the one of computing marginal probabilities,
and this is done using recent deterministic algorithms for computing such marginals in certain models satisfying
SSM property~\cite{weitzCounting},\cite{GamarnikKatz},\cite{BayatiGamarnikKatzNairTetali}. These results are
based on establishing even stronger property, namely correlation decay on a computation tree and lead to efficient
algorithms for computing such marginals. The original intention of these papers was constructing scalable
(polynomial time) approximation algorithms for computing partition functions on general finite graphs,
but an intermediate problem of computing approximately marginal probabilities is solved.
We have implemented these algorithm for the hard-core~\cite{weitzCounting}
and monomer-dimer~\cite{BayatiGamarnikKatzNairTetali} models in the special case of $\Z^d$. Using our approach
we improve existing bounds for these models. For example we show that the exponent of the free-energy
of the hard-core model in two dimensions is in the range $[1.503034,1.503058]$. While our lower
bound is weaker than the previous best known bound $1.503047782$~\cite{CalkinWilf} (see above), which is
already quite close to a believed estimate~\cite{Baxter99},
our upper bound improves the earlier best bound $1.5035148$~\cite{CalkinWilf}.
For the case of monomer-dimer model in three dimensions we obtained a range
$[0.78595,0.78599]$, substantially improving earlier bounds~\cite{FriedlandPeled},\cite{FriedlandGurvits},
\cite{FriedlandKropLundowMarkstrom} (see above). Further we show
that the numerical complexity of obtaining $\epsilon$-additive approximation using our approach
is $(1/\epsilon)^{O(1)}$ for \emph{both} free energy and surface pressure. The constant
in $O(1)$ may depend on model parameters and dimensions. For example, for the case of
monomer-dimer model, this constant is $Cd^{1\over 2}\log d$, where $C$ is some universal constant
(Proposition~\ref{prop:ComputationalComplexityM}).
This is a substantial
improvement over the computation efforts $\exp(O((1/\epsilon)^{d-1})$ and $\exp(O((1/\epsilon)^{2d-2})$
of the transfer matrix method for computing free energy and surface pressure.

The rest of the paper is organized as follows. In the following section we provide a necessary background on Gibbs measures
on general graphs and lattices, free energy, surface pressure, and define an important notion -- (exponential) Strong
Spatial Mixing. Our main theoretical result is Theorem~\ref{theorem:MainResultGeneral}
which represents free energy and surface pressure in terms of marginal probabilities. This result and its
several variations are stated and proven
in Section~\ref{section:general}. Sections~\ref{section:IS} and~\ref{section:matchings} are devoted to application
of Theorem~\ref{theorem:MainResultGeneral} and its variations to the problem of numerically estimating
free energy for hard-core and monomer-dimer models specifically. Additionally, in these sections
we compare the algorithmic complexity
of our method with the complexity of the  transfer matrix method. Concluding thoughts and open questions
are in Section~\ref{section:conclusions}.

\section{Model, assumptions and notations}\label{section:model}

\subsection{Finite and locally finite graphs}
Consider a finite or infinite locally finite simple undirected  graph $\G$
with node set $V$ and edge set $E\subset V\times V$. The locally-finite property means
every node is connected to only finitely many neighbors.
The graph is undirected and simple (no loops, no multiple edges).
For every $v,u\in V$ let
$d(u,v)$ be the length (number of edges) in the shortest path connecting $u$ and $v$.
For every node $v$, we let $N(v)$ stand for the set of neighbors of $v$: $N(v)=\{u: (u,v)\in E\}$.
We will write $N_{\G}(v)$ when we need to emphasize the underlying graph $\G$.
The quantity $\Delta=\Delta_{\G}\triangleq \max_v |N(v)|$ is called the degree of the graph.
For every $r\in \Z_+$ let  $B_r(v)=\{u:d(v,u)\le r\}$.
For every set $A\subset V$, let $B_r(A)=\cup_{v\in A}B_r(v)$. Thus locally finite property means
$|B_r(v)|<\infty$ for all $v,r$.
Given $A\subset V$, let $\partial A=\{v\in A: N(v)\cap A^c\ne \emptyset\}$ and let
$\partial_r A=\partial (B_r(A))$.
For every  subset $A\subset V$, we have an induced subgraph obtained by taking nodes in $A$
and all edges $(v,u)\in E\cap A^2$. The corresponding edge set of the induced graph is denoted by $E(A)$.

Let
$\R_+$ ($\Z_+$) denote the set of all non-negative real (integer) values. Le
$\R_{>0} (\Z_{>0})$ be the set of all positive real (integer) values.
Our main example of an infinite locally-finite graph
is the $d$-dimensional lattice $\Z^d$ with $V=\{(v_1,\ldots,v_d): v_i\in \Z$ and
$E=\{(v,u)\in V^2: \|v-u\|=1\}$, where $\|w\|=\sum_{1\le j\le d}|w_j|$.
We denote the origin $(0,0,\ldots,0)$ by $0$ for short.
Given a vector $a=(a_1,\ldots,a_d)\in \R_+^d$,
and $v=(v_1,\ldots,v_d)\in \Z^d$,
let $B_{an}(v)=\{u\in \Z^d: |u_j-v_j|\le a_jn, ~j=1,2,\ldots,d\}$.
In the special case $\G=\Z^d, v=0$, we write $B_r$ instead of $B_r(0)$ and $B_{an}$ instead of $B_{an}(0)$.
For each $j\le d, k\in \Z_+$, let $\Z^d_{j,k,+}=\{v\in\Z^d:v_j\le k\},\Z^d_{j,k,-}=\{v\in\Z^d:v_j\ge -k\}$.
Each of these is a $d$-dimensional half-plane in $\Z^d$.
Let $\prec$ denote a lexicographic full order on $\Z^d$. Namely, $v\prec u$ iff either $u_d>v_d$ or
$\exists k\in \{1,2,\ldots,d-1\}$ such that
$u_k>v_k$ and $v_j=u_j, j=k+1,\ldots,d$.
For every
$v\in \Z^d$ and $k=1,2,\ldots,d$, let
$\Z^d_{\prec v}=\{u\in\Z^d:u\prec v\}\cup \{v\}.$ Similarly we define $\Z^d_{j,k,+,\prec v}$ and $\Z^d_{j,k,-,\prec v}$
with $\Z^d_{j,k,+}$ and $\Z^d_{j,k,-}$ replacing $\Z^d$.

Throughout the paper we write $f(n)=O(g(n))$ and $f(n)=o(g(n)),~n\in \Z_+$ if $f(n)\le Cg(n)$,
respectively $f(n)/g(n)\rightarrow 0$, for all $n$,
for some constant $C$. This constant may in general depend on model parameters
such as $H,h$ (see the next section) or dimension $d$.
However, in some places the constant is universal, namely independent from any model parameters. We will explicitly
say so if this is the case.

\subsection{Gibbs measures}
Consider a finite set of spin values $\chi=\{s_1,\ldots,s_q\}$, a Hamiltonian
function $H:\chi^2\rightarrow \R\cup \{\infty\}$, and an
external field $h:\chi\rightarrow\R$.
Given a graph $\G=(V,E)$ we consider the associated spin configuration space $\Omega=\chi^{|V|}$ equipped
with product $\sigma$-field $\F$. If the graph $\G$ is finite, a
probability measure $\pr$ on $(\Omega,\F)$ is defined to be Gibbs measure
if for every spin assignment $(s_v)\in \chi^V$
\begin{align}
&\pr\left(\sigma_v=s_v,~\forall~ v\in V\right)
=Z^{-1}\exp\big(-\sum_{(v,u)\in E} H(s_v,s_u)-\sum_{v\in V}h(s_v)\big), \label{eq:Gibbs}
\end{align}
where $Z$ is the normalizing partition function:
\begin{align*}
Z=\sum_{(s_v)\in\chi^{|V|}}\exp\big(-\sum_{(v,u)\in E} H(s_v,s_u)-\sum_{v\in V}h(s_v)\big).
\end{align*}
We will often write $\pr_{\G}$ and $Z_{\G}$ in order to emphasize the underlying graph.
The case $H(a,a')=\infty$ corresponds to a hard-core constraint prohibiting assigning $a\in\chi$ and $a'\in
\chi$ to neighbors.
The possibility of such hard-core constraints is important for us when we discuss the problems of counting
independent sets and matchings.

When $\G$ is infinite, a probability measure $\pr(\cdot)$ is defined to be Gibbs measure, if it
satisfies the following spatial Markovian property. For every finite $A\subset V$ and every
spin assignment $(s_v)\in \chi^{|\partial A|}$ on the boundary $\partial A$ of $A$,
the conditional probability measure $\pr(\cdot|(s_v))$ on the finite graph $(V(A),E(A))$ induced by $A$
is a Gibbs measure with the same $H$ as the original graph and external field $h'$ given as
$h'_v(a)=h(a)+\sum_{u\in N(v)\setminus A}h_u(s_u)$ for every $a\in \chi$. Here every node
$u\in N(v)\setminus A$ belongs to $\partial A$ and thus its spin value $s_u$ is well defined.
This is called spatial Markovian property of Gibbs measure.
One can construct Gibbs measures
as a weak limit of Gibbs measures on cylinder sets in $\F$ obtained from finite induced subgraphs
 of $\G$, see \cite{SimonLatticeGases},\cite{GeorgyGibbsMeasure}
for details. Generally there are multiple Gibbs measures and the
space of Gibbs measures of $\G$ is denoted by $\M$ or $\M_\G$. We say that the model $(\G,H,h)$ is in the uniqueness
regime if $\M$ consists of a unique measure $\M=\{\pr\}$. Most of the results in this paper correspond to the uniqueness case,
and more specifically to the case of Strong Spatial Mixing (SSM) defined below.
Given a model $(\G,H,h)$ and a subset $A\subset V$, we have a naturally defined induced submodel on the
the induced subgraph $\G(A)=(V(A),E(A))$, given by the same $\chi,H$ and $h$ (but different partition function).
In order to emphasize the underlying subgraph $\G$ we write,
with abuse of notation,
$\Z_{A}$ and $\pr_{A}$ for the partition function and the  Gibbs measure on the subsystem
$(\G(A),H,h)$, when it is unique.
By default we drop the subscript when the underlying graph is the entire lattice $\Z^d$.

The definitions above is "node" centered: the spins are associated with nodes of a graph. In order to study the monomer-dimer
model we need to consider a similar "edge" model where spins are associated with edges. Thus given a finite graph $\G=(V,E)$,
and a finite set of spin values $\chi$, we consider a probability space $\Omega=\chi^{|E|}$. Then (\ref{eq:Gibbs})
is restated as follows. Write $e\sim e'$ if edges $e,e'$ are distinct and incident (share a node).
For every spin assignment $(s_e)\in \chi^{|E|}$ we assign probability measure
\begin{align}
&\pr\left(\sigma_e=s_e,~\forall~ e\in E\right)
=Z^{-1}\exp\big(-\sum_{e\sim e'\in E} H(s_e,s_{e'})-\sum_{e\in E}h(s_e)\big), \label{eq:GibbsEdges}
\end{align}
It is clear that "edge" model can be reduced to "node" model by considering a line graph of $\G$: the nodes of this graph
are edges of $\G$ and two nodes $e,e'$ form an edge in the line graph if and only if  $e\sim e'$.
For simplicity, however, we prefer not to switch to the line graph model.

\subsection{Free energy, pressure and surface pressure}
Given  $\G=\Z^d$, Hamiltonian $H$ and an external field $h$,  consider an arbitrary infinite
sequence of finite subsets $0\in \Lambda_1\subset\Lambda_2\subset\cdots\subset \Z^d$, such that the sequence
$r_n\triangleq\max\{r: B_r\subset \Lambda_i\}$ diverges to infinity as $n\rightarrow\infty$. Consider
the corresponding sequence of Gibbs measures $\pr_{\Lambda_n}$ and partition functions $Z_{\Lambda_n}$
on the graphs induced by $\Lambda_n$.
It follows from
sub-additivity property of partition functions that the limit
\begin{align}
\mathcal{P}(d,H,h)\triangleq\lim_{n\rightarrow\infty}{\log Z_{\Lambda_n}\over |\Lambda_n|} \label{eq:pressure}
\end{align}
exists and is independent from the choice of the sequence of subsets~\cite{SimonLatticeGases}, \cite{GeorgyGibbsMeasure}.
This quantity is called \emph{pressure}. Given a positive $\beta>0$ one usually also considers limits
$\lim_{n\rightarrow\infty}\log Z_{\Lambda_n}/(\beta |\Lambda_n|)$ where $H$ and $h$ are replaced by $\beta H$ and
$\beta h$, respectively. The corresponding limit is called \emph{free energy}. For our purposes, this difference
is insubstantial, as  it is just a matter
of redefining $H,h$ and changing the normalization. We will mostly use the term \emph{free energy}.

The first order correction to the limit (\ref{eq:pressure}) is known as \emph{surface pressure} and
is defined as follows. Unlike pressure, this quantity  is "shape" dependent.
In this paper we only consider the case of the rectangular shape, although our results can extended to more complicated
shapes as well. Given a vector $a=(a_i)\in \R^d_{>0}$,
let  $A_j(a)\triangleq\prod_{k\ne j}(2a_k), ~j=1,2,\ldots,d,$ and let $A(a)=\sum_{j\le d}A_j(a)$.
Observe that surface area of $B_{an}$ (the number of boundary nodes in $B_{an}$) is $A(a)n^{d-1}+o(n^{d-1})$.
The surface pressure is defined as
\begin{align}
s\mathcal{P}&(d,H,h,a)=\notag\\
&=\lim_{n\rightarrow\infty} A^{-1}(a)n^{-d+1}\left(\log Z_{B_{an}(0)}-
n^d\mathcal{P}(d,H,h)\prod_{1\le i\le d}(2a_i)\right). \label{eq:surfacepressure}
\end{align}
The surface pressure is interpreted as the first order correction of the free energy, scaled by the
surface area. The existence of this limit is a classical fact~\cite{SimonLatticeGases} for the case
$\max_{s,s'\in\chi}H(s,s')<\infty$. The case hard-core case $H=\infty$ is trickier and we are not aware
of results on the existence of the limit for this case. Our proofs however imply the existence of the limit
under the assumption of SSM, see below.

\subsection{Strong Spatial Mixing}\label{subsection:SSM}
The following technical assumption is needed for our analysis.
The assumption essentially says that there is always a choice of a spin value $s^*\in\chi$
which allows any other choice of spin values for its neighbors without making the Hamiltonian infinite.
\begin{Assumption}\label{assumption:s*}
There exist $s^*\in \chi$ such that  $\max_{s\in\chi} |H(s,s^*)|<\infty$.
\end{Assumption}
\noindent
The assumption  implies that under any Gibbs measure, for every node $v$
\begin{align}\label{eq:c^*}
c^*\triangleq \min_{(s_u)}\pr(\sigma_v=s^*|\sigma_u=s_u, \forall u\in N(v))>0
\end{align}
where the minimum is over all possible spin assignments
$(s_u)\in \chi^{|N(v)|}$ of the neighbors of $v$.
We now introduce the Strong Spatial Mixing assumption,
which, in particular, implies the uniqueness of the Gibbs measure on an infinite graph.
We first informally discuss this notion. Loosely speaking, a model $(\G,H,h)$ exhibits just \emph{spatial mixing}
if for every finite set $A\subset \G$ the joint probability law for spins in $A$ is asymptotically independent
from the values of spins which are far away  from $A$. A form of this condition is known
to be equivalent to the uniqueness of the Gibbs measure. Instead, a strong spatial mixing means the same
except that values for any other subset $B\subset \Z^d$ are allowed to be fixed. In other words, strong spatial
mixing is spatial mixing, but for the \emph{reduced} model on $V\setminus B$ when values of spins in $B$ are fixed
and values of the external field $h$ on the remaining graph are appropriately modified. Strong spatial mixing is strictly stronger
than spatial mixing (hence the separate definition), see~\cite{BertoinMartinelliPeres}
for counterexamples and discussion.

We now introduce the definition formally.
\begin{Defi}\label{definition:SSM}
$(\G,H,h)$ satisfies the Strong Spatial Mixing property if for every finite set $X\subset V(\G)$,
possibly infinite set $Y\in V(\G)$
there exists a function $R(r)$ satisfying $\lim_{r\rightarrow \infty}R(r)=0$, such that
for every positive integer $r$
\begin{align*}
\max
\Big|&\pr(\sigma_v=s^1_v,v\in X|\sigma_v=s^2_v,v\in Y; \sigma_v=s^3_v,v\in\partial_r X) \\
&-\pr(\sigma_v=s^1_v,v\in X|\sigma_v=s^2_v,v\in Y; \sigma_v=s^4_v,v\in\partial_r X) \Big| \le R(r),
\end{align*}
where the maximum is over all possible spin assignments $(s^1_v)\in \chi^{|X|},(s^2_v)\in \chi^{|Y|},
(s^3_v),(s^4_v)\in \chi^{|\partial_r X|}$.
$(\G,H,h)$ satisfies exponential strong spatial mixing if
there exist $\kappa,\gamma>0$ such that $R(r)\le \kappa\exp(-\gamma r)$ for all  $r\ge 0$.
\end{Defi}

It is known that if $(\G,H,h)$ exhibits SSM, then the Gibbs measure is  unique:
$\mathcal{M}=\{\pr\}$,~\cite{BertoinMartinelliPeres}.

\section{Sequential cavity method}\label{section:general}
In this section we present our main theoretical result: the representation of the free energy and surface pressure
on $\Z^d$ in terms of some conditional marginal probabilities $\pr(\sigma_0=s|\cdot)$ defined on
suitably modified subsets of $\Z^d,\Z_{j,k,+},\Z_{j,k,-}$. The idea is to sequentially remove nodes from a
rectangle $B_{an}$ one by one and observe that the log-partition function can be written in terms of
the log-marginal probabilities (cavity) of the removed nodes.

\subsection{Representation theorem}
Given a graph $\G=(V,E)$ with some full order $\succ$, a node $v$ and $s\in \chi$,
let $\mathcal{E}_{v,s}$ denote the event "$\sigma_u=s$ for all $u\succ v, u\in V$".
By convention we assume that $\mathcal{E}_{v,s}$ is the full event $\Omega$, if the set $u\succ v, u\in V$ is empty.
In the special case when $\G=\Z^d, v=0$ and $\succ$ is the lexicographic order,
this event is denoted by $\mathcal{E}_s$, see Figure~\ref{fig:CavityGeneral} for the case $d=2$.
We remind the reader, that we drop the
subscripts in $\pr_{\G},Z_{\G}$ when the underlying graph is $\Z^d$ and SSM holds guaranteeing uniqueness
of the Gibbs measure.

\begin{theorem}\label{theorem:MainResultGeneral}
Suppose  $(\Z^d,H,h)$ satisfies the SSM property and the Assumption~\ref{assumption:s*} holds.
Then
\begin{align}\label{eq:PressureGeneral}
\mathcal{P}(d,H,h)=-\log \pr(\sigma_{0}=s^*|\mathcal{E}_{s^*})-dH(s^*,s^*)-h(s^*).
\end{align}
If, in addition, SSM is exponential, then for every $a=(a_j)\in \R_{>0}^d$
\begin{align}\label{eq:SurfacePressureGeneral}
s\mathcal{P}(d,H,h,a)&=-H(s^*,s^*)+ \\
&\sum_{1\le j\le d}{A_j(a)\over A(a)}\sum_{k=0}^\infty\Big(2\log \pr(\sigma_{0}=s^*|\mathcal{E}_{s^*})-
\log \pr_{\Z^d_{j,k,+}}(\sigma_{0}=s^*|\mathcal{E}_{s^*}) \notag\\
&-\log \pr_{\Z^d_{j,-k,-}}(\sigma_{0}=s^*|\mathcal{E}_{s^*})\Big) \notag.
\end{align}
Moreover, the infinite sum is convergent.
\end{theorem}

\remark We stress that our proof does not rely on the existence of a limit in (\ref{eq:pressure}) and thus
provides an independent proof for it, albeit in the restricted case of SSM.
Preempting the formal discussion, the convergence of the infinite sum in (\ref{eq:SurfacePressureGeneral})
is an almost immediate consequence of exponential SSM, which gives
\begin{align*}
|\pr_{\Z^d_{j,k,+}}(\sigma_{0}=s^*|\mathcal{E}_{s^*})-\pr(\sigma_{0}=s^*|\mathcal{E}_{s^*})|
\le \kappa\exp(-\gamma k).
\end{align*}
A similar conclusion for difference of the logarithms is obtained then using Assumption~\ref{assumption:s*}.
This observation implies that $\epsilon$-additive approximation of this infinite
sum is obtained by considering the partial sum with first $O(\log(1/\epsilon))$ terms.
As we shall discuss
in the following subsection,
the main import from this observation is that the overhead of computing surface pressure vs free energy is very small.

\begin{figure}
\begin{center}
\includegraphics[scale=.3]{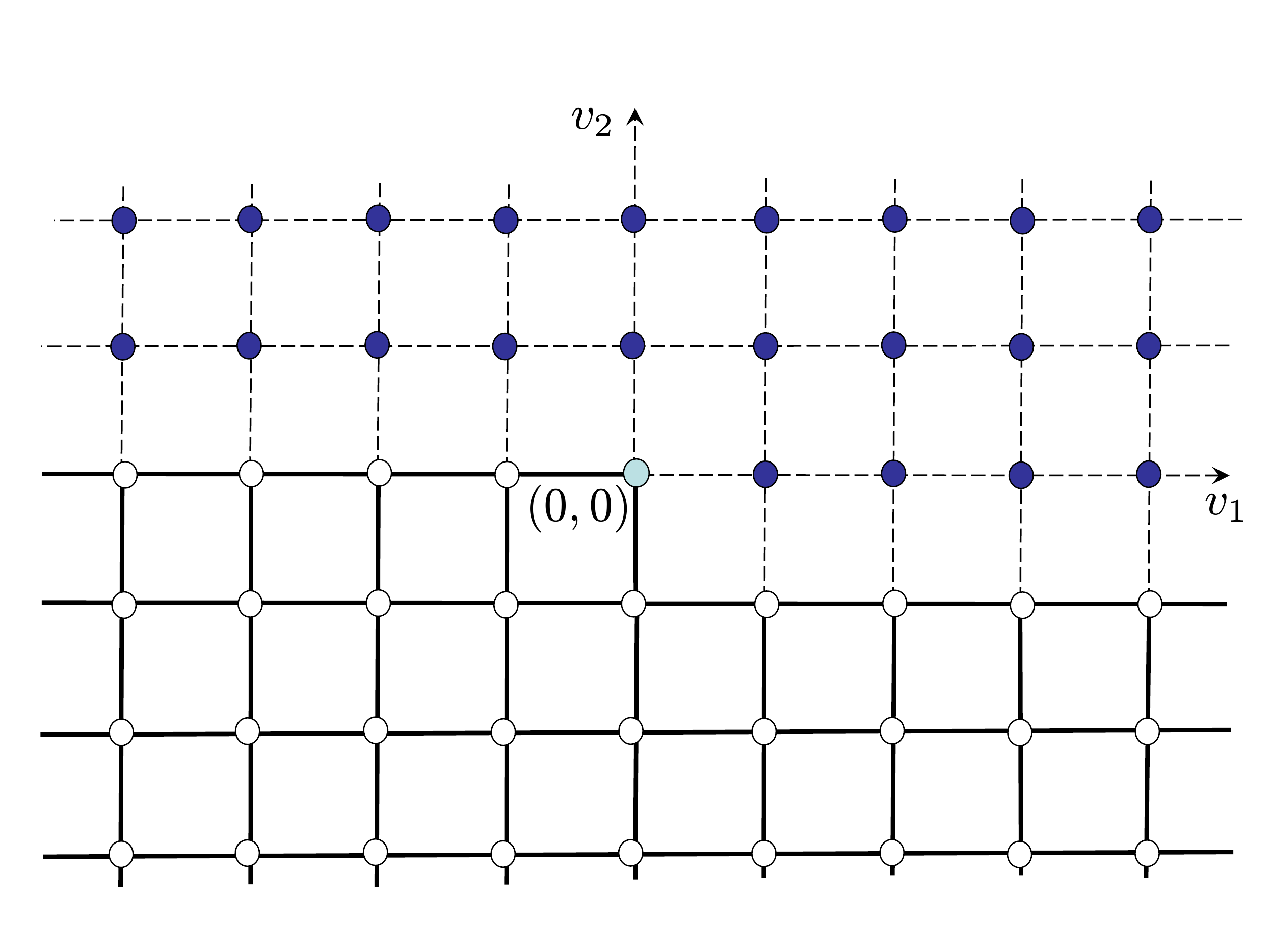}
\caption{Event $\mathcal{E}_{s^*}$ on $\Z^2$. Every dark node is assigned spin $s^*$}\label{fig:CavityGeneral}
\end{center}
\end{figure}

\begin{proof}
We first prove (\ref{eq:PressureGeneral}).
Let  $v_1\succ v_2\succ \ldots\succ v_{|B_n|}$ be the labeling of nodes in $B_n$ according to the
lexicographic order. Note that the number of edges in $B_n$ is  $d|B_n|+o(|B_n|)$. We have
\begin{align*}
\pr_{B_n}&(\sigma_{v}=s^*, \forall v\in B_n)\\
&=Z_{B_n}^{-1}\exp\big(-d|B_n|H(s^*,s^*)-o(|B_n|)H(s^*,s^*)-|B_n|h(s^*)\big),
\end{align*}
from which we infer that
\begin{align}
Z_{B_n}^{-1}
&=\exp\big(d|B_n|H(s^*,s^*)+o(|B_n|)H(s^*,s^*)+|B_n|h(s^*)\big)\times \notag\\
&\times\pr_{B_n}(\sigma_{v}=s^*, \forall v\in B_n). \label{eq:ZN-1}
\end{align}
On the other hand, by telescoping property
\begin{align}\label{eq:telescope}
\pr_{B_n}(\sigma_{v}=s^*, \forall v\in B_n)&=
\prod_{v\in B_n}\pr_{B_n}(\sigma_{v}=s^*|\sigma_{u}=s^*,~\forall~ u\succ v, u\in B_n)\\
&=\prod_{v\in B_n}\pr_{B_n}(\sigma_{v}=s^*|\mathcal{E}_{v,s^*}) \notag
\end{align}

Fix $\epsilon>0$ and find $r=r(\epsilon)$ such that according to Definition~\ref{definition:SSM}, $R(r)<\epsilon$
for $X=\{0\}, Y=\{u\succ 0\}$.
Let $B^o=\{v\in B_n: B_r(v)\subset B_n\}$.
Observe that $|B^o|/|B_n|\rightarrow 1$ as $n\rightarrow\infty$.
By the choice of $r$ we have for every $v\in B^o$
\begin{align}\label{eq:MarginalFiniteInfinite}
\Big|\pr(\sigma_{v}=s^*|\mathcal{E}_{v,s^*})
-\pr_{B_n}(\sigma_{v}=s^*|\mathcal{E}_{v,s^*})\Big|\le \epsilon.
\end{align}
By translation invariance we have
\begin{align*}
\pr(\sigma_{v}=s^*|\mathcal{E}_{v,s^*})
=
\pr(\sigma_{0}=s^*|\mathcal{E}_{s^*})
\end{align*}
For every $v\in B_n\setminus B^o$ we have the  generic lower bound (\ref{eq:c^*})
\begin{align}\label{eq:lowerNotinBo}
\pr_{B_n}(\sigma_{v}=s^*|\mathcal{E}_{v,s^*})
\ge c^*
\end{align}
which is strictly positive by Assumption~\ref{assumption:s*}.
A similar inequality with the same constant holds for
$\pr(\sigma_{0}=s^*|\mathcal{E}_{s^*})$.
Combining this with (\ref{eq:ZN-1}) and (\ref{eq:telescope}) we obtain
\begin{align*}
Z_{B_n}^{-1}&=\exp\big(d|B_n|H(s^*,s^*)+o(|B_n|)H(s^*,s^*)+|B_n|h(s^*)\big)\times\\
&\times\prod_{v\in B^o}\pr_{B_n}(\sigma_{v}=s^*|\mathcal{E}_{v,s^*})
\prod_{v\in B_n\setminus B^o}\pr_{B_n}(\sigma_{v}=s^*|\mathcal{E}_{v,s^*}) \\
&\ge\exp\big(d|B_n|H(s^*,s^*)+o(|B_n|)H(s^*,s^*)+|B_n|h(s^*)\big) \\
&\big(\pr(\sigma_{0}=s^*|\mathcal{E}_{s^*})-\epsilon)^{|B^o|}
(c^*)^{|B_n\setminus B^o|}
\end{align*}
Since $|B^o|/|B_n|\rightarrow 1$ as $n\rightarrow\infty$ and $c^*>0$, then we obtain
\begin{align*}
\liminf_{n\rightarrow\infty}{\log Z_{B_n}^{-1}\over |B_n|}&\ge dH(s^*,s^*)+h(s^*)+
\log(\pr(\sigma_{0}=s^*|\mathcal{E}_{s^*})-\epsilon)
\end{align*}
Recalling  $\pr(\sigma_{0}=s^*|\mathcal{E}_{s^*})\ge c^*>0$,
since $\epsilon$ was arbitrary, we conclude
\begin{align*}
\liminf_{n\rightarrow\infty}{\log Z_{B_n}^{-1}\over |B_n|}&\ge H(s^*,s^*)d+h(s^*)+
\log \pr(\sigma_{0}=s^*|\mathcal{E}_{s^*}).
\end{align*}
Similarly we show
\begin{align*}
\limsup_{n\rightarrow\infty}{\log Z_{B_n}^{-1}\over |B_n|}&\le H(s^*,s^*)d+h(s^*)+
\log \pr(\sigma_{0}=s^*|\mathcal{E}_{s^*}),
\end{align*}
where for the case $v\in B_n\setminus B_n^o$ we use a trivial inequality
$\pr_{B_n}(\sigma_{v}=s^*|\mathcal{E}_{v,s^*})\le 1$
in place of (\ref{eq:lowerNotinBo}). We obtain
\begin{align}\label{eq:limsupinf}
\lim_{n\rightarrow\infty}{\log Z_{B_n}^{-1}\over |B_n|}=H(s^*,s^*)d+h(s^*)+\log \pr(\sigma_{0}=s^*|\mathcal{E}_{s^*}),
\end{align}
This concludes the proof of (\ref{eq:PressureGeneral}).

Now we establish (\ref{eq:SurfacePressureGeneral}). Thus consider a rectangle $B_{an}$.
The proof is based on a more refined estimates for the
elements of the telescoping product (\ref{eq:telescope}) with $B_{an}$ replacing $B_n$.
We begin by refinement of (\ref{eq:ZN-1}).
The number of edges in $B_{an}$ is
$d|B_{an}|-A(a)n^{d-1}+o(n^{d-1})$. Repeating the derivation of (\ref{eq:ZN-1}) and (\ref{eq:telescope}), we obtain
\begin{align}
Z_{B_{an}}^{-1}
&=\exp\big(d|B_{an}|H(s^*,s^*)-A(a)n^{d-1}H(s^*,s^*)+h(s^*)|B_{an}|+o(n^{d-1})\big) \notag\\
&\times \prod_{v\in B_{an}}\pr_{B_{an}}(\sigma_{v}=s^*|\mathcal{E}_{v,s^*}) \label{eq:ZN-1a}
\end{align}
Let
\begin{align*}
B^o_{an}&=\Big\{v\in B_{an}: |a_in-v_i|,|a_in+v_i|\ge C\log n, 1\le i\le d \Big\},
\end{align*}
and for each $j=1,\ldots,d$ and $0\le k< C\log n$ let
\begin{align*}
B_{an,j,k,+}&=\Big\{v\in B_{an}:v_j=\lfloor a_jn\rfloor-k; |a_in-v_i|,|a_in+v_i|\ge C\log n, i\ne j\Big\},\\
B_{an,j,k,-}&=\Big\{v\in B_{an}:v_j=-\lfloor a_jn\rfloor +k; |a_in-v_i|,|a_in+v_i|\ge C\log n, i\ne j\Big\}.
\end{align*}
Here $C>0$ is a large yet unspecified constant.
In other words, $B_{an,j,k,+}\cup B_{an,j,k,-}$ consists of nodes
which have distance $k$ from from the boundary of $B_{an}$ in the $j$ coordinate, but have distance
at least $C\log n$ in all the other coordinates $i\ne j$.
Observe that the sets $B_{an,j,k,+}, B_{an,j,k,-}$ are non-intersecting, for sufficiently large $n$.
Observe also that
\begin{align}\label{eq:B0a}
\Big|B_{an}&\setminus \big(B^o_{an}\cup \cup_{1\le j\le d,0\le k\le C\log n}
(B_{an,j,k,+}\cup B_{an,j,k,-})\big)\Big| \notag\\
&=O(n^{d-2}\log^2 n ).
\end{align}
For every $v\in B^o_{an}$ we have by exponential SSM
\begin{align*}
\big|\pr(\sigma_{v}=s^*|\mathcal{E}_{v,s^*})
-\pr_{B_{an}}(\sigma_{v}=s^*|\mathcal{E}_{v,s^*})\big|\le \kappa\exp(-\gamma C\log n)=O(n^{-\gamma C}).
\end{align*}
By translation invariance we have
\begin{align*}
\pr(\sigma_{v}=s^*|\mathcal{E}_{v,s^*})=
\pr(\sigma_{0}=s^*|\mathcal{E}_{s^*})
\end{align*}
Recalling $\pr(\sigma_{0}=s^*|\mathcal{E}_{s^*})\ge c^*$, we obtain using the Taylor expansion,
\begin{align}\label{eq:InternalNodes}
\big|\log\pr_{B_{an}}(\sigma_{v}=s^*|\mathcal{E}_{v,s^*})-\log\pr(\sigma_{0}=s^*|\mathcal{E}_{s^*})\big|=O(n^{-\gamma C}).
\end{align}

Fix $v\in B_{an,j,k,+}$. Since $v_j=\lfloor a_jn\rfloor-k$ and for every $i\ne j$,
$v_i$ is at least $C\log n$ away from the boundary of $B_{an}$,
then by the exponential SSM property we have
\begin{align*}
\Big|\pr_{B_{an}}(\sigma_{v}=s^*|\mathcal{E}_{v,s^*})-
\pr_{\Z^d_{j,\lfloor a_jn\rfloor,+}}(\sigma_{v}=s^*|\mathcal{E}_{v,s^*})
\Big|\le \kappa\exp(-\gamma C\log n)=O(n^{-\gamma C}),
\end{align*}
(see Figure~\ref{fig:Cavity-Level-K} for the two-dimensional illustration with $j=2$).
Recalling $v_j=\lfloor a_jn\rfloor-k$ and using translation invariance we have
\begin{align*}
\pr_{\Z^d_{j,\lfloor a_jn\rfloor,+}}(\sigma_{v}=s^*|\mathcal{E}_{v,s^*})=
\pr_{\Z^d_{j,k,+}}(\sigma_{0}=s^*|\mathcal{E}_{s^*})
\end{align*}

\begin{figure}
\begin{center}
\includegraphics[scale=.4]{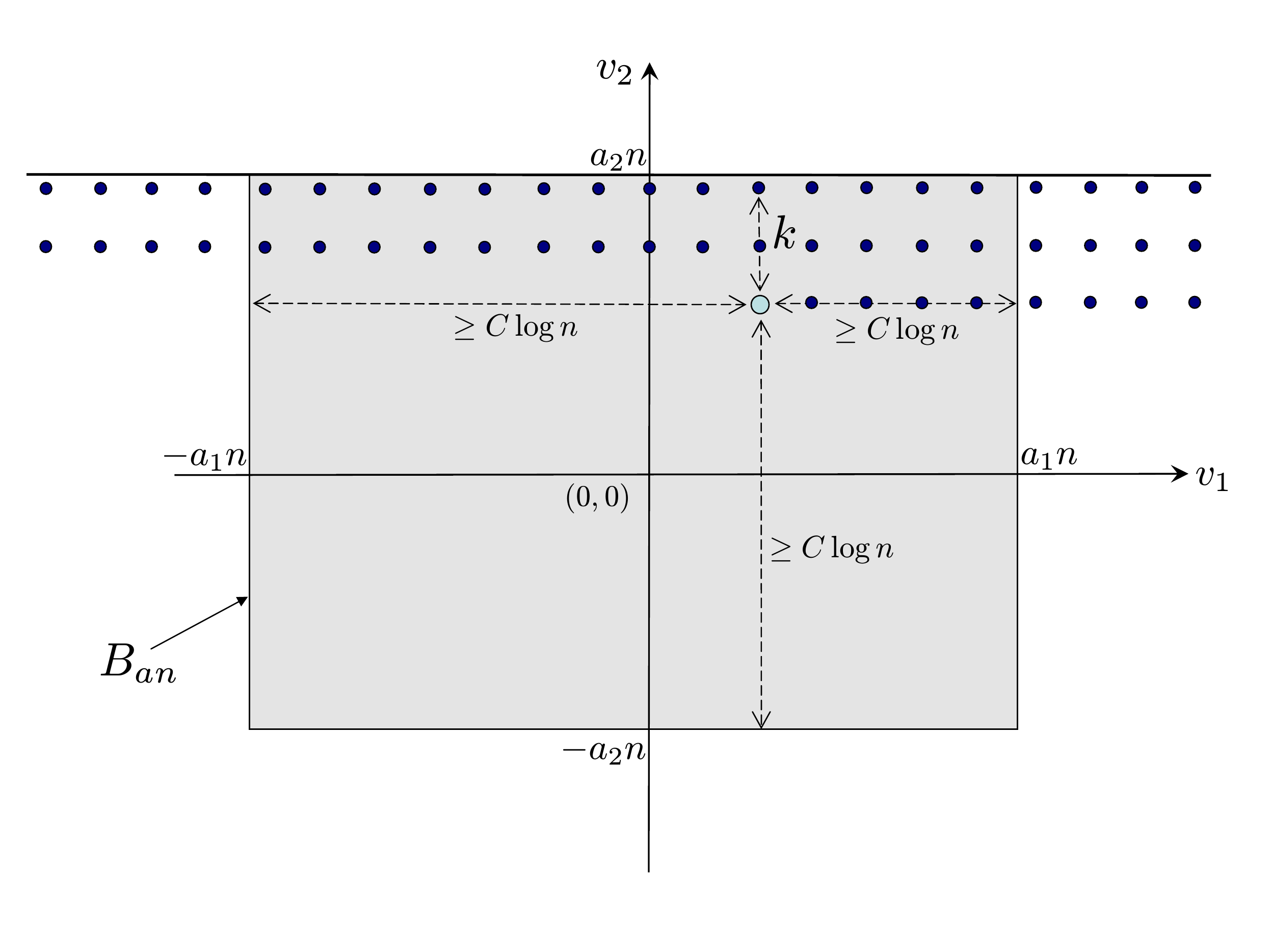}
\caption{Event $\mathcal{E}_{s^*}$ on $\Z^2_{2,\lfloor a_2n\rfloor,+}$.
Every dark node is assigned spin $s^*$}\label{fig:Cavity-Level-K}
\end{center}
\end{figure}
Using again the Taylor expansion, we obtain
\begin{align}\label{eq:B+approx}
\Big|\log\pr_{B_{an}}(\sigma_{v}=s^*|\mathcal{E}_{v,s^*})-
\log\pr_{\Z^d_{j,k,+}}(\sigma_{0}=s^*|\mathcal{E}_{s^*})\Big|=O(n^{-\gamma C}).
\end{align}
Similarly, if $v\in B_{an,j,k,-}$, then
\begin{align*}
\Big|\pr_{B_{an}}(\sigma_{v}=s^*|\mathcal{E}_{v,s^*})-
\pr_{\Z^d_{j,-\lfloor a_jn\rfloor,-}}(\sigma_{v}=s^*|\mathcal{E}_{v,s^*})
\Big|\le \exp(-\gamma C\log n)=O(n^{-\gamma C}).
\end{align*}
By translation invariance and since $v_j=-\lfloor a_jn\rfloor+k$, then
\begin{align*}
\pr_{\Z^d_{j,-\lfloor a_jn\rfloor,-}}(\sigma_{v}=s^*|\mathcal{E}_{v,s^*})=
\pr_{\Z^d_{j,-k,-}}(\sigma_{0}=s^*|\mathcal{E}_{s^*}),
\end{align*}
and again applying Taylor expansion
\begin{align}\label{eq:B-approx}
\Big|\log\pr_{B_{an}}(\sigma_{v}=s^*|\mathcal{E}_{v,s^*})-
\log\pr_{\Z^d_{j,-k,-}}(\sigma_{0}=s^*|\mathcal{E}_{s^*})\Big|=O(n^{-\gamma C}).
\end{align}
We now take $\log$ of both sides of (\ref{eq:ZN-1a}) and divide by $n^{d-1}$ to obtain
\begin{align*}
{-\log Z_{B_{an}}+|B_{an}| \mathcal{P}(d,H,h)\over n^{d-1}}
&={|B_{an}|\over n^{d-1}} \mathcal{P}(d,H,h)+{|B_{an}|\over n^{d-1}}dH(s^*,s^*) \\
&- A(a)H(s^*,s^*)+
{|B_{an}|\over n^{d-1}}h(s^*)+o(1) \notag\\
&+n^{-d+1}\sum_{v\in B^o_{an}}\log\pr_{B_{an}}(\sigma_{v}=s^*|\mathcal{E}_{v,s^*})\\
&+n^{-d+1}\sum_{1\le j\le d}\sum_{k\le C\log n}\sum_{v\in B_{an,j,k,+}}\log\pr_{B_{an}}(\sigma_{v}=s^*|\mathcal{E}_{v,s^*})\\
&+n^{-d+1}\sum_{1\le j\le d}\sum_{k\le C\log n}\sum_{v\in B_{an,j,k,-}}\log\pr_{B_{an}}(\sigma_{v}=s^*|\mathcal{E}_{v,s^*})
\end{align*}
Here we use (\ref{eq:B0a}) and the fact $\pr_{B_{an}}(\sigma_{v}=s^*|\mathcal{E}_{v,s^*})\ge c^*>0$ for all $v$.
Applying (\ref{eq:InternalNodes})
\begin{align*}
\sum_{v\in B^o_{an}}&\log\pr_{B_{an}}(\sigma_{v}=s^*|\mathcal{E}_{v,s^*})=\\
&=|B^o_{an}|\log\pr(\sigma_{0}=s^*|\mathcal{E}_{s^*})+|B^o_{an}|O(n^{-\gamma C})\\
&=|B_{an}|\log\pr(\sigma_{0}=s^*|\mathcal{E}_{s^*})-
\log\pr(\sigma_{0}=s^*|\mathcal{E}_{s^*})\times \\
&\sum_{1\le j\le d}\sum_{k\le C\log n}(|B_{an,j,k,+}|+|B_{an,j,k,+}|)
+o(n^{d-1}),
\end{align*}
where we used (\ref{eq:B0a}) in the last inequality and assume that $C>1/\gamma$.
Using the established identity (\ref{eq:PressureGeneral}) we conclude
\begin{align*}
&{-\log Z_{B_{an}}+|B_{an}| \mathcal{P}(d,H,h)\over n^{d-1}}
=- A(a)H(s^*,s^*)+ \\
&+n^{-d+1}\sum_{1\le j\le d}\sum_{k\le C\log n}\sum_{v\in B_{an,j,k,+}}
\big(\log\pr_{B_{an}}(\sigma_{v}=s^*|\mathcal{E}_{v,s^*})-\log\pr(\sigma_{0}=s^*|\mathcal{E}_{s^*})\big)\\
&+n^{-d+1}\sum_{1\le j\le d}\sum_{k\le C\log n}\sum_{v\in B_{an,j,k,-}}\big(\log\pr_{B_{an}}(\sigma_{v}=s^*|\mathcal{E}_{v,s^*})
-\log\pr(\sigma_{0}=s^*|\mathcal{E}_{s^*})\big) \\
&+o(1)
\end{align*}
Further applying (\ref{eq:B+approx}) and (\ref{eq:B-approx}) and using  $C>1/\gamma$,
we obtain the following expression
\begin{align*}
&- A(a)H(s^*,s^*)+ \\
&+n^{-d+1}\sum_{1\le j\le d}\sum_{k\le C\log n}|B_{an,j,k,+}|
\big(\log\pr_{\Z^d_{j,k,+}}(\sigma_{0}=s^*|\mathcal{E}_{v,s^*})
-\log\pr(\sigma_{0}=s^*|\mathcal{E}_{s^*})\big)\\
&+n^{-d+1}\sum_{1\le j\le d}\sum_{k\le C\log n}|B_{an,j,k,-}|
\big(\log\pr_{\Z^d_{j,-k,-}}(\sigma_{0}=s^*|\mathcal{E}_{v,s^*})
-\log\pr(\sigma_{0}=s^*|\mathcal{E}_{s^*})\big) \\
&+o(1)
\end{align*}
We have
\begin{align*}
|B_{an,j,k,+}|=A_j(a)n^{d-1}+o(n^{d-1}), ~~|B_{an,j,k,-}|=A_j(a)n^{d-1}+o(n^{d-1}).
\end{align*}
The resulting expression is then
\begin{align*}
&- A(a)H(s^*,s^*)+\\
&+\sum_{1\le j\le d}A_j(a)\sum_{k\le C\log n}
\big(\log\pr_{\Z^d_{j,k,+}}(\sigma_{0}=s^*|\mathcal{E}_{v,s^*})
-\log\pr(\sigma_{0}=s^*|\mathcal{E}_{s^*})\big)\\
&+\sum_{1\le j\le d}A_j(a)\sum_{k\le C\log n}
\big(\log\pr_{\Z^d_{j,-k,-}}(\sigma_{0}=s^*|\mathcal{E}_{v,s^*})
-\log\pr(\sigma_{0}=s^*|\mathcal{E}_{s^*})\big)
+o(1)
\end{align*}
Applying the exponential SSM property and the Taylor expansion,  we have for every $k\ge C\log n$
\begin{align*}
&\big|\log\pr_{\Z^d_{j,k,+}}(\sigma_{0}=s^*|\mathcal{E}_{v,s^*})
-\log\pr(\sigma_{0}=s^*|\mathcal{E}_{s^*})\big)| \le \kappa\exp(-\gamma k)\\
&\big|\log\pr_{\Z^d_{j,-k,-}}(\sigma_{0}=s^*|\mathcal{E}_{v,s^*})
-\log\pr(\sigma_{0}=s^*|\mathcal{E}_{s^*})\big)|\le \kappa\exp(-\gamma k).
\end{align*}
This means that we can replace the sums $\sum_{k\le C\log n}$ with infinite sums $\sum_{k\ge 0}$,
with a resulting error $O(\exp(-\gamma C\log n)=O(n^{-\gamma C})$.
Dividing  by $A(a)$ we obtain the result.
\end{proof}

\subsection{Extensions and variations}
In this subsection we establish several variations of the first part of Theorem~\ref{theorem:MainResultGeneral},
namely the representation (\ref{eq:PressureGeneral}) of the free energy in terms of the marginal probability.
First we extend identity (\ref{eq:PressureGeneral}) for the case when we do not necessarily
have SSM, but instead have an upper or lower bound on marginal probability $\pr(\sigma_{0}=s^*|\mathcal{E}_{s^*})$.
In this case we obtain an analogue of (\ref{eq:PressureGeneral}) in the form of inequalities. These
inequalities will be useful for obtaining numerical bounds on free energy for hard-core model in dimensions
$d=3,4$ in Section~\ref{section:IS}.

Given $r>0$ consider an arbitrary spin assignment $(s_u)\in \chi^{|\partial B_r|}$ which is consistent
with event $\mathcal{E}_{s^*}$. Namely, $s_u=s^*$ for all $u\succ 0, u\in \partial B_r$.
Let $p_{\max}(r)~(p_{\min}(r))$ be the maximum (minimum) of $\pr_{B_r}(\sigma_{0}=s^*|\mathcal{E}_{s^*},(a_u))$
when we vary over all such assignments.

\begin{coro}\label{coro:MainResultUpperLower}
For every $(\Z^d,H,h)$ and $r\ge 0$
\begin{align}\label{eq:PressureGeneralUpperLower}
-\log p_{max}(r)&-dH(s^*,s^*)-h(s^*)\notag\\
&\le \mathcal{P}(d,H,h)\notag\\
&\le -\log p_{min}(r)-dH(s^*,s^*)-h(s^*).
\end{align}
\end{coro}
While the result holds for arbitrary $r$ the quality of the bounds presumably improves with increasing $r$.
We will see in Section~\ref{section:IS} that in some cases $p_{\max}(r)$ and $p_{\min}(r)$ are fairly close for large $r$
even though the model is outside of provably exponential SSM regime.

\begin{proof}
The proof is a minor variation of the proof of (\ref{eq:PressureGeneral}).
Instead of estimate (\ref{eq:MarginalFiniteInfinite}) we use
$p_{\min}(r)\le \pr_{B_n}(\sigma_{v}=s^*|\mathcal{E}_{v,s^*})\le p_{\max}(r)$ for every $n>r$.
\end{proof}

Our second variation is a "chess-pattern" version of Theorem~\ref{theorem:MainResultGeneral}
where in representation (\ref{eq:telescope})
we sequentially remove only vertices with even sum of coordinates.
As it turns out this
version provides substantial gains in computing numerical estimates of free energy both for hard-core
and monomer-dimer models, though we do not have theoretical explanation for this gain.

Let $\Z^d_{\even}=\{v=(v_1,\ldots,v_d)\in \Z^d: \sum_i v_i~\text{is even}\}$. Similarly
define $\Z^d_{\odd}$.
Given a subgraph $\G=(V,E)$ of $\Z^d$, let
$\mathcal{E}_{v,s,\even}$ denote the event $\sigma_u=s, \forall u\succ v, u\in V\cap\Z^d_{\even}$.
The special case $v=0$ is denoted by $\mathcal{E}_{s,\even}$. See Figure~\ref{fig:CavityGeneralChessPattern}
for this chess-pattern
version of the event $\mathcal{E}_{s^*}$.

\begin{figure}
\begin{center}
\includegraphics[scale=.3]{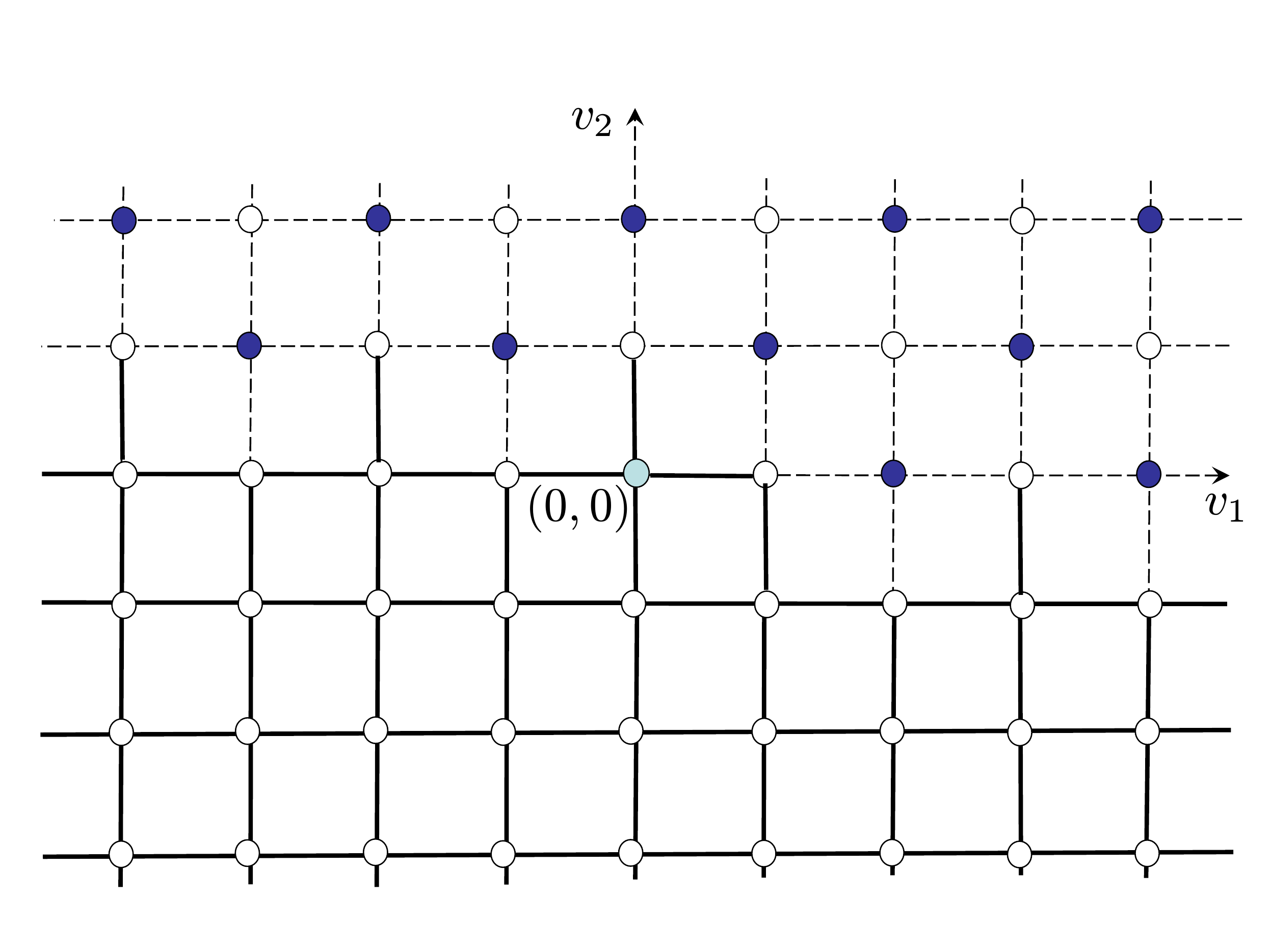}
\caption{Event $\mathcal{E}_{s^*,\even}$ on $\Z^2$. Every dark node is assigned spin $s^*$}
\label{fig:CavityGeneralChessPattern}
\end{center}
\end{figure}

\begin{theorem}\label{theorem:MainResultGeneralChessPattern}
Suppose  $(\Z^d,H,h)$ satisfies the SSM property and the Assumption~\ref{assumption:s*} holds.
Then
\begin{align}\label{eq:PressureGeneralChessPattern}
\mathcal{P}(d,H,h)
&=-{1\over 2}\log \pr(\sigma_{0}=s^*|\mathcal{E}_{s^*,\even}) \notag\\
&+{1\over 2}\log\Big(\sum_{s\in \chi}\exp(-2dH(s,s^*)-h(s))\Big).
\end{align}
\end{theorem}

\begin{proof}
Let
$B_{n,\even}=B_n\cap \Z^d_{\even}$. Observe that every node $v\in B_n\setminus B_{n,\even}$
has only neighbors in $B_{n,\even}$. Let $\Delta(v)$ denote the degree of $v$ in $B_n$.
As a result,
\begin{align*}
\pr_{B_n}(\sigma_{v}&=s^*, \forall v\in B_{n,\even})=
Z_{B_n}^{-1}\prod_{v\in B_n\setminus B_{n,\even}}\sum_{s\in \chi}\exp(-\Delta(v)H(s,s^*)-h(s)),
\end{align*}
from which we obtain
\begin{align*}
Z_{B_n}^{-1}=\prod_{v\in B_n\setminus B_{n,\even}}\Big(\sum_{s\in \chi}\exp(-\Delta(v)H(s,s^*)-h(s))\Big)^{-1}
\pr_{B_n}(\sigma_{v}=s^*, \forall v\in B_{n,\even}).
\end{align*}
On the other hand, by telescoping property
\begin{align}\label{eq:telescopeChessPattern}
\pr_{B_n}(\sigma_{v}=s^*, \forall v\in B_{n,\even})&=
\prod_{v\in B_{n,\even}}\pr_{B_n}(\sigma_{v}=s^*|\sigma_{u}=s^*,~\forall~ u\succ v, u\in B_{n,\even})\\
&=\prod_{v\in B_{n,\even}}\pr_{B_n}(\sigma_{v}=s^*|\mathcal{E}_{v,s^*,\even}) \notag
\end{align}
The remainder of the proof is very similar to the proof of Theorem~\ref{theorem:MainResultGeneral}
and details are omitted. Notice that for $|B_n|-o(|B_n|)$ nodes in $B_n$, the degree $\Delta(v)=2d$.
Then in place of (\ref{eq:limsupinf}) we obtain
\begin{align*}
\lim_n{\log Z_{B_n}^{-1}\over |B_n|}&=
-\lim_{n\rightarrow\infty}{|B_n\setminus B_{n,\even}|\over |B_n|}
\log\Big(\sum_{s\in \chi}\exp(-2dH(s,s^*)-h(s))\Big) \\
&+\lim_{n\rightarrow\infty}{| B_{n,\even}|\over |B_n|}\log \pr(\sigma_{0}=s^*|\mathcal{E}_{s^*,\even})\\
&=-{1\over 2}\log\Big(\sum_{s\in \chi}\exp(-2dH(s,s^*)-h(s))\Big)
+{1\over 2}\log \pr(\sigma_{0}=s^*|\mathcal{E}_{s^*,\even})
\end{align*}

\end{proof}

Now let us present a version of Theorem~\ref{theorem:MainResultGeneral} for the model (\ref{eq:GibbsEdges})
where spins are assigned to edges rather than nodes. We need this for application to the monomer-dimer model.
Let $\mathcal{E}_{v,s,\text{edges}}$ denote the event $\sigma_{u,w}=s, \forall u\succ v, w\in N(u)$,
and let $\mathcal{E}_{v,s,\text{edges},\even}$ denote the event $\sigma_{u,w}=s, \forall u\succ v, u\in\Z^d_{\even}, w\in N(u)$.
Let $\mathcal{E}_{s,\text{edges}},\mathcal{E}_{s,\text{edges},\even}$  denote the same events when $v=0$.

\begin{theorem}\label{theorem:MainResultGeneralEdges}
Consider a model $(\Z^d,H,h)$ given by (\ref{eq:GibbsEdges}) with spins assigned to edges.
Suppose  $(\Z^d,H,h)$ satisfies the SSM property and the Assumption~\ref{assumption:s*} holds.
Then
\begin{align}
\mathcal{P}(d,H,h)
&=-\log \pr(\sigma_{(0,v)}=s^*, \forall v\in N(0)|\mathcal{E}_{s^*,\text{edges}})\notag\\
&-d(2d-1)H(s^*,s^*)-dh(s^*)
\label{eq:PressureGeneralEdges}\\
&=-{1\over 2}\log \pr(\sigma_{(0,v)}=s^*, \forall v\in N(0)|\mathcal{E}_{s^*,\text{edges},even})\notag\\
&-d(2d-1)H(s^*,s^*)-dh(s^*).
\label{eq:PressureGeneralEdgesChessPattern}
\end{align}
\end{theorem}

\remark Contrast this result with Theorem~\ref{theorem:MainResultGeneralChessPattern}.
There the extra term $\log(\sum_{s\in\chi}\cdot)$ appears since every odd node surrounded by even nodes with
preset spin value $s^*$ still has $\chi$ choices for the spin selection. For the edge version this is not the case:
all edges are preselected to take spin values $s^*$.

\begin{proof}
We first prove (\ref{eq:PressureGeneralEdges}).
Let $E(B_n)$ denote the edge set of $B_n$. Note that $E(B_n)=d|B_n|+o(|B_n|)$ and the the number
of edges with $2d-2$ incident edges is also $d|B_n|+o(|B_n|)$. The number of pairs of incident
edges is then $d(2d-1)|B_n|+o(|B_n|)$.
We have
\begin{align*}
\pr_{B_n}(\sigma_{e}&=s^*, \forall e\in E(B_n))\\
&=Z_{B_n}^{-1}\exp\big(-d(2d-1)|B_n|H(s^*,s^*)-o(|B_n|)H(s^*,s^*)-d|B_n|h(s^*)\big).
\end{align*}
On the other hand, by telescoping property
\begin{align*}
\pr_{B_n}&(\sigma_{e}=s^*, \forall e\in E(B_{n}))\\
&=\prod_{v\in B_{n}}\pr_{B_n}(\sigma_{(v,u)}=s^*, \forall u\in N(v)|
\sigma_{(u,w)}=s^*,~\forall~ u\succ v, w\in N(u),~ u,w\in B_{n}) \\
&=\prod_{v\in B_{n}}\pr_{B_n}(\sigma_{(v,u)}=s^*, \forall u\in N(v)|\mathcal{E}_{v,s^*,\text{edges}}).
\end{align*}
The remainder of the proof of
(\ref{eq:PressureGeneralEdges}) is similar to the one of Theorem~\ref{theorem:MainResultGeneral} and is omitted.

Turning to (\ref{eq:PressureGeneralEdgesChessPattern}), observe that every edge in $\Z^d$ has exactly one end point
in $\Z^d_{\even}$. Then
we have again by the telescoping property
\begin{align*}
\pr_{B_n}&(\sigma_{e}=s^*, \forall e\in E(B_{n}))\\
&=\prod_{v\in B_{n,\even}}\pr_{B_n}(\sigma_{(v,u)}=s^*, \forall u\in N(v)|
\sigma_{(u,w)}=s^*,~\forall~ u\succ v, w\in N(u),~ u\in B_{n,\even},w\in B_{n}) \\
&=\prod_{v\in B_{n,\even}}\pr_{B_n}(\sigma_{(v,u)}=s^*, \forall u\in N(v)|\mathcal{E}_{v,s^*,\text{edges},\even}).
\end{align*}
The remainder of the proof of
(\ref{eq:PressureGeneralEdges}) is similar to the one of Theorem~\ref{theorem:MainResultGeneral}.
The fact $|B_{n,even}|/|B_n|\rightarrow 1/2$ as $n\rightarrow\infty$ leads to a factor $1/2$ in
(\ref{eq:PressureGeneralEdgesChessPattern}).
\end{proof}

\subsection{Applications and numerical complexity}\label{subsection:complexity}
Theorem~\ref{theorem:MainResultGeneral} reduces the problem of computing $\mathcal{P}$ and $s\mathcal{P}$
to the problem of computing conditional marginal probabilities $\pr(\sigma_0=s^*|\mathcal{E}_{s^*})$. This is certainly
not the only way to represent free energy and surface pressure in terms of marginal probabilities.
For example, consider a modified system $(\G,\beta H,\beta h)$ on a finite graph $\G$, and observe that

\begin{align*}
{d\log Z_{\G}\over d\beta}=-\sum_{v\in V}\E[h(\sigma_v)]-\sum_{(v,u)\in E}\E[H(\sigma_v,\sigma_u)],
\end{align*}
both expectations are taken wrt the Gibbs measure $\pr$. Thus knowing the marginal probabilities $\pr(\sigma_v)$
and joint probabilities $\pr(\sigma_v,\sigma_u), (v,u)\in E$
lets us recover $\log Z_{\G}$ approximately in principle.
Unfortunately, this means we have to integrate the answers over $\beta\in [0,1]$,
which in practice has to be approximated by summation. In order then to guarantee the target level of accuracy, one
would have to control the derivatives of marginal probabilities wrt $\beta$.
Additionally,
one would have to compute marginal probabilities for a whole range of $\beta$, over which the integration takes place.
The advantage of the representation (\ref{eq:PressureGeneral}) and (\ref{eq:SurfacePressureGeneral}) is that it allows
us computing free energy and surface pressure by computing \emph{only one}
marginal probability $\pr(\sigma_0=s^*|\mathcal{E}_{s^*})$.
In this paper we will compute these marginal probabilities approximately using recent deterministic algorithms for computing
such marginal probabilities in general graphs, where appropriately defined computation tree (see following chapters) satisfies
exponential SSM. As we will see below, when we have such property, our method provides an additive $\epsilon$
approximation of free energy in time $(1/\epsilon)^{O(1)}$, where the constant $O(1)$
may depend on the model parameters and $d$. Let us now compare this performance with the performance of the
transfer matrix method. While we are not aware of any systematic numerical complexity analysis of the transfer
matrix method, it can be deduced from the following considerations. The transfer matrix method is based on first computing
the partition function on a strip $[-n,n]^{d-1}\times \Z$. The latter is done by constructing certain
$|\chi|^{(2n+1)^{d-1}}$ by $|\chi|^{(2n+1)^{d-1}}$ transfer matrix.
For the hard-core case with $\lambda=1$ (see Section~\ref{section:IS})
the matrix is $0-1$ with $1$ corresponding to allowed pair of neighboring configurations and $0$ corresponding
to pairs of configurations which are not allowed. The spectral radius of the transfer matrix is then used to deduce
the growth rate of the partition function, namely the free energy. Constructing such matrix takes time $\exp(O(n^{d-1}))$.
Since the partition function on
$[-n,n]^{d-1}\times\Z$ converges to the one of $\Z^d$ at the rate $O(1/n)$~\cite{SimonLatticeGases},
then in order to achieve an additive error
$\epsilon>0$, one needs $\exp(O((1/\epsilon)^{d-1}))$ computation effort.

Suppose one then wishes to use this method to approximate the surface pressure with an additive error $\epsilon$.
What is the required numerical effort?
We are not aware of applications of the transfer matrix method for computing the surface pressure. Thus
the reasonable alternative approach is to
1) select a rectangle $an$ for some large value $n$, 2) Compute
the partition function $Z_{an}$ in this rectangle using perhaps the brute force method and 3) Compute
the approximation $\hat P$ of the free energy $\mathcal{P}$ using perhaps the transfer matrix method.
The approximate surface pressure is then obtained from the fact that the
convergence rate in the limit (\ref{eq:surfacepressure}) is $O(1/n)$~\cite{SimonLatticeGases}.
Let us show that this approach
requires $\exp\big(O((1/\epsilon)^{2d-2})\big)$ numerical effort in order to obtain $\epsilon$ additive
approximation, regardless of how quickly one is able to compute $Z_{an}$.  Indeed,
observe that this approach leads to an
error $n^d|\hat P-\mathcal{P}|/n^{d-1}=n|\hat P-\mathcal{P}|$.
 Thus for the target $\epsilon$ additive error,
we need to set $n\ge 1/\epsilon$. But this further requires that
$(1/\epsilon)|\hat P-\mathcal{P}|=O(\epsilon)$, namely $|\hat P-\mathcal{P}|=O(\epsilon^2)$.
Thus we need to achieve $\epsilon^2$ additive error accuracy in estimating the free energy.
This requires $\exp\big(O((1/\epsilon)^{2d-2})\big)$ per our earlier calculations,
and the assertion is established.
This is a stark contrast with complexity $(1/\epsilon)^{O(1)}$ of the method proposed in this paper
both for free energy and surface pressure for the cases of hard-core and monomer-dimer models.
However, one should note that our method takes explicit advantage of the exponential SSM, while the
transform matrix method does not rely on this assumption.

\section{Hard-core (independent set)  model}\label{section:IS}
The hard-core lattice gas model, commonly known as independent set model in combinatorics, is given by
$\chi=\{0,1\}, H(0,0)=H(0,1)=H(1,0)=0, H(1,1)=\infty, h(0)=0,h(1)=\beta$ for some parameter $\beta$.
Case of interest is  $\beta\le 0$ as it corresponds to Gibbs measure putting larger weight on larger cardinality
independent set. Choosing $s^*=0$ we obtain that Assumption~\ref{assumption:s*} holds.

It is common to set $\lambda=\exp(-\beta)>0$ and let $\lambda$ be the parameter of the
hard-core model. This parameter is usually called \emph{activity}. Note that in terms of $\lambda$, for every
finite graph $\G$
\begin{align}\label{eq:IndSetPartitionFunction}
Z_{\G}=\sum\lambda^{|\{v\in V:\sigma_v=1\}|}
\end{align}
where the sum is over all spin configurations $(\sigma_v)\in \{0,1\}^V$ such that
$\sigma_v\sigma_u=0$ for all $(v,u)\in E$.
Equivalently, a subset of nodes  $I\subset V(\G)$ is called an independent (also sometimes called a stable set)
if for no edge $(u,v)$ we have both $u$ and $v$ belong to $I$. Then we may rewrite (\ref{eq:IndSetPartitionFunction}) as
\begin{align}
Z_{\G}=\sum_{I}\lambda^{|I|} \label{eq:PartitionFunctionIS}
\end{align}
where the summation is over all independent sets of $\G$. The summation in (\ref{eq:PartitionFunctionIS})
is sometimes called an independent set polynomial in the combinatorics literature.
From now on we let $\mbI$ denote the random independent set selected according to a Gibbs measure
(multiply defined if there are many Gibbs measures).
In the case of finite graph $\G$, for every independent set $I$ we have
\begin{align*}
\pr_{\G}(\mbI=I)=Z_{\G}^{-1}\lambda^{|I|}.
\end{align*}
The special case $\lambda=1$ corresponds to a uniform distribution on the set of all independent sets in $\G$.
We denote the free energy and surface pressure on $\Z^d$ by $\mathcal{P}(d,\lambda)$ and $s\mathcal{P}(d,\lambda,a)$,
respectively. In the special case $\lambda=1$, the free energy is also the entropy of the Gibbs distribution, since it is
uniform.
Conditioning on spins taking value $s^*=0$ simplifies significantly in the context of hard-core model as the following
proposition shows, a simple proof of which we include for completeness.

\begin{prop}\label{prop:hardcorereduction}
If a hard-core model on a graph $\G=(V,E)$ satisfies SSM for some $\lambda$, then so does
any subgraph of $\G$. The same assertion applies to exponential SSM.
Moreover, for every $W_1,W_2\subset V(\G)$ the following identity holds
with respect to the unique Gibbs measure.
\begin{align}\label{eq:reductionIS}
\pr_{\G}(v\in \mbI| W_1\cap\mbI=\emptyset, W_2\subset \mbI)&=\pr_{\hat\G}(v\in\mbI),
\end{align}
for every $v\in V\setminus (W_1\cup B_1(W_2))$,
where $\hat\G$ is the subgraph induced by nodes in $V\setminus (W_1\cup B_1(W_2))$.
\end{prop}

\begin{proof}
Fix any positive integer $r$ and consider any spin assignment
$(s_u), u\in \partial B_r(v)\setminus (W_1\cup B_1(W_2))$. Extend this to a spin assignment to entire
$\partial B_r(v)$ by setting $s_u=1$ (that is $u\in \mbI$), for $u\in \partial B_r(v)\cap W_2$ and $s_u=0$ (that is $u\notin\mbI$)
for $u\in \partial B_r(v)\cap \Big(W_1\cup (B_1(W_2)\setminus W_2)\Big)$.
Call this spin assignment $\mathcal{S}$.
Applying spatial Markovian property we have
\begin{align*}
\pr_{\G}(\sigma_v=0&|\sigma_u=0, u\in W_1,\sigma_u=1,u\in W_2, \mathcal{S}) \\
&={\pr_{\G}(\sigma_v=0,\sigma_u=0, u\in W_1,\sigma_u=1,u\in W_2,| \mathcal{S})
\over \pr_{\G}(\sigma_u=0, u\in W_1,\sigma_u=1,u\in W_2|\mathcal{S})} \\
&={\sum_{I\in \mathcal{I}_1}\lambda^{|I|}\over \sum_{I\in \mathcal{I}_2}\lambda^{|I|}},
\end{align*}
where $\mathcal{I}_1$ is the set of independent sets in $B_r(v)$ such that
$v\notin I, I\cap W_1=\emptyset, W_2\subset I$ and $u\in I\cap \partial B_r(v)$ iff $s_u=1$ (according to $\mathcal{S}$).
The set $\mathcal{I}_2$ is defined similarly, except the condition $v\notin I$ is dropped.
Then the ratio is equal to
\begin{align*}
{\sum_{I\in \mathcal{I}_1}\lambda^{|I|-|W_2|}\over \sum_{I\in \mathcal{I}_2}\lambda^{|I|-|W_2|}}=
{\sum_{I\in \mathcal{I}_3}\lambda^{|I|}\over \sum_{I\in \mathcal{I}_4}\lambda^{|I|}}
\end{align*}
where $\mathcal{I}_4$ is the set of independent subsets of $B_r(0)\setminus (W_1\cup B_1(W_2))$
such that $u\in I\cap \partial B_r(v)$ iff $s_u=1$, and $\mathcal{I}_3$ is defined similarly, except
in addition $v\notin I$ for every $I\in \mathcal{I}_3$. We recognize this ratio as
$\pr_{\hat \G}(\sigma_v=0|\mathcal{S})$. We conclude
\begin{align*}
\pr_{\G}(\sigma_v=0|\sigma_u=0, u\in W_1,\sigma_u=1,u\in W_2, \mathcal{S})=\pr_{\hat \G}(\sigma_v=0|\mathcal{S}).
\end{align*}
Now consider any induced subgraph $\G_1$ of $\G$ and any two sets $W_1,W_2$ in $\G_1$.
Applying a similar argument we obtain
\begin{align*}
\pr_{\G_1}&(\sigma_v=0|\sigma_u=0, u\in W_1,\sigma_u=1,u\in W_2, \mathcal{S}) \\
&=\pr_{\G}(\sigma_v=0|\sigma_u=0, u\in W_1\cup (V\setminus V_1),\sigma_u=1,u\in W_2, \mathcal{S}) \\
&=\pr_{\hat \G_1}(\sigma_v=0|\mathcal{S})
\end{align*}
where $\hat\G_1$ is induced by nodes $V_1\setminus (W_1\cup B_1(W_2))$. By SSM property, the second quantity
has a limit as $r\rightarrow\infty$ which is independent from assignment $\mathcal{S}$ on the boundary
$\partial B_r(v)$. Therefore the same applies to the first and third quantities. The first conclusion implies that
$\G_1$ satisfies SSM. The second conclusion gives (\ref{eq:reductionIS}) when applied to $\G_1=\G$.
\end{proof}

In light of Propoposition~\ref{prop:hardcorereduction}
we obtain the following simplification of
Theorems~\ref{theorem:MainResultGeneral}  and  \ref{theorem:MainResultGeneralChessPattern} in the hard-core case. Let
\begin{align*}
\Z^d_{\prec 0,\even}=\Z^d_{\prec 0}\cup \{u\in \Z^d_{\odd}: 0\prec u\},
\end{align*}
see Figure~\ref{fig:CavityGeneralChessPattern}.
\begin{coro}\label{corollary:IS}
Suppose the hard-core model on $\Z^d$ satisfies SSM for a given $\lambda$. Then
\begin{align}
\mathcal{P}(d,\lambda)&=-\log \pr_{\Z^d_{\prec 0}}(0\notin \mbI) \label{eq:PressureIS1}\\
&=-{1\over 2}\log \pr_{\Z^d_{\prec 0, \even}}(0\notin \mbI)+{1\over 2}\log (1+\lambda). \label{eq:PressureIS1chesspattern}
\end{align}
\end{coro}
Thus we now focus on developing an algorithm for numerically estimating marginal probabilities appearing
in (\ref{eq:PressureIS1}) and (\ref{eq:PressureIS1chesspattern}).

\subsection{Recursion, sequential cavity algorithm and correlation decay}\label{subsection:AlgorithmIS}
Let us now introduce a recursion satisfied by the hard-core model.
This identity in a different form using a self-avoiding tree construction was
established recently by  Weitz~\cite{weitzCounting}. We repeat here some of the developments
in~\cite{weitzCounting}, with some minor modifications, which
are indicated as necessary.

\begin{theorem}\label{theorem:RecursionIndSet}
Given a finite graph $\G=(V,E)$ and  $v\in V$, let $N(v)=\{v_1,\ldots,v_k\}$. Then
\begin{align}\label{eq:RecursionIndSet}
\pr_{\G}(v\notin \mbI)={1\over 1+\lambda\prod_{1\le i\le k}\pr_{\G_{i-1}}(v_i\notin \mbI)}
\end{align}
where $\G_i$ is the graph induced by $V\setminus \{v,v_1,\ldots,v_{i}\}$, $\G_0$ is induced
by $V\setminus \{v\}$ and $\prod_{1\le i\le k}=1$ when $k=0$.
\end{theorem}

\begin{proof}
We have
\begin{align}\label{eq:recursionIndSetZ}
Z_{\G}&=\sum_{I: v\notin I}\lambda^{|I|}+\sum_{I: v\in I}\lambda^{|I|}
=\sum_{I: I\subset V\setminus \{v\}}\lambda^{|I|}+
\lambda\sum_{I: I\subset V\setminus \{v,v_1,\ldots,v_k\} }\lambda^{|I|}
\end{align}
where everywhere the sums are over independent sets $I$. Note that
$\sum_{I: I\subset V\setminus \{v\}}\lambda^{|I|}=Z_{\G_0}$
and\\
$\sum_{I: I\subset V\setminus \{v,v_1,\ldots,v_k\}}\lambda^{|I|}=Z_{\G_k}$.
Dividing both sides of the identity (\ref{eq:recursionIndSetZ}) by $Z_{\G_0}$ we obtain
\begin{align*}
{Z_{\G}\over Z_{\G_0}}=1+\lambda {Z_{\G_k}\over Z_{\G_0}}
\end{align*}
It is immediate that $Z_{\G_0}/Z_{\G}=\pr_{\G}(v\notin \mbI)$. In order to interpret
$Z_{\G_k}/Z_{\G_0}$ similarly we rewrite it as
\begin{align*}
{Z_{\G_k}\over Z_{\G_0}}=\prod_{i=1}^k{Z_{\G_i}\over Z_{\G_{i-1}}}
\end{align*}
and note that $Z_{\G_i}/Z_{\G_{i-1}}=\pr_{\G_{i-1}}(v_i\notin \mbI)$.
Combining these observations we obtain (\ref{eq:RecursionIndSet}).

\end{proof}

The identity (\ref{eq:RecursionIndSet}) suggests a recursion for computing marginal probabilities
$\pr_{\G}(v\notin\mbI)$ approximately. The idea is to apply the identity recursively several times
and then set the initial values arbitrarily. One then establishes further that a correlation decay property
holds on this recursion which implies that any initialization of the values at the beginning of the recursion
leads to approximately correct values at the end of the recursion. This principle was underlying the algorithm
proposed in~\cite{weitzCounting} for computing approximately the number of independent sets in general
graphs. The original approach taken in~\cite{weitzCounting} was slightly different - first a self-avoiding
tree corresponding to the recursive computation tree described above is constructed. Then it is shown
that  $\pr(v\notin I)$ on this tree equals the same probability in the underlying graph.
The approach proposed here is slightly simpler as it bypasses the extra argument of showing equivalence
of two marginal probabilities.

We now provide details of this approach.
Given a finite graph $\G$, for every subgraph $\hat\G=(\hat V,\hat E)$ of $\G$, every vertex $v\in
\hat V$ and every  $t\in \mathbb{Z}_+$ we introduce a quantity
$\Phi_{\hat\G}(v,t)$ defined inductively as follows.
\begin{align}\label{eq:Phirecursion}
\Phi_{\hat \G}(v,t)=\left\{
                      \begin{array}{ll}
                        1, & \hbox{$t=0$;} \\
                        (1+\lambda)^{-1}, & \hbox{$t>0, ~N(v)=\emptyset,$} \\
                        (1+\prod_{1\le i\le k}\Phi_{\hat\G_{i-1}}(v_i,t-1))^{-1}, &
                    \hbox{$t>0, ~N(v)=\{v_1,\ldots,v_k\}\ne\emptyset$.}
                    \end{array}
                    \right.
\end{align}
Here again $\hat\G_0$ is induced by $\hat V\setminus \{v\}$ and $\hat\G_i$ is induced by
$\hat V\setminus \{v,v_1,\ldots,v_i\}$.
The recursion (\ref{eq:Phirecursion}) is naturally related to the identity (\ref{eq:RecursionIndSet}).
Specifically if $\Phi_{\hat\G_{i-1}}(v_i,t-1)=\pr_{\hat\G_{i-1}}(v_i\notin \mbI)$ for all $i$
then $\Phi_{\hat\G}(v,t)=\pr_{\hat\G}(v\notin \mbI)$. However, this will not occur in general,
as we set $\Phi_{\hat \G}(v,0)=1$, due to the lack of knowledge of actual values of the corresponding
probabilities.

Similarly to $\Phi$, we introduce values $\Psi_{\hat G}(v,t)$ with the only exception that
$\Psi_{\hat \G}(v,0)=0$ for all $\hat\G=(\hat V,\hat E)$ and $v\in\hat V$. The following lemma follows
from Theorem~\ref{theorem:RecursionIndSet} and the definitions of $\Phi$ and $\Psi$
using a simple induction argument
\begin{lemma}\label{lemma:UpperLowerIS}
For every $v$ and $t\in\Z_+$
\begin{align*}
\Psi_{\G}(v,2t)&\le \pr_{\G}(v\notin \mbI)\le \Phi_{\G}(v,2t), \\
\Phi_{\G}(v,2t+1)&\le \pr_{\G}(v\notin \mbI)\le \Psi_{\G}(v,2t+1)
\end{align*}
\end{lemma}
Next we provide bounds on the computational effort required to compute $\Phi$ and $\Psi$.
\begin{lemma}\label{lemma:ComputationalComplexityISGeneral}
For every finite graph $\G$ with degree $\le\Delta$, $v\in\G$ and $t$, the values
$\Phi_{\G}(v,t),\Psi_{\G}(v,t)$ can be computed in time $\exp(O(t\log \Delta))$,
where the constant in $O(\cdot)$ is universal.
\end{lemma}

\begin{proof}
The result follows immediately from the recursive definitions of $\Phi$ and $\Psi$.
\end{proof}

The crucial correlation decay property is formulated in the following proposition.
\begin{theorem}[\cite{weitzCounting}]\label{theorem:CorrelationDecayIS}
For every  $\Delta\ge 3$ and for every
\begin{align}\label{eq:lambdaBound}
\lambda<(\Delta-1)^{\Delta-1}/(\Delta-2)^{\Delta},
\end{align}
there exists $C=C(\lambda,\Delta),\rho=\rho(\lambda,\Delta)<1$ such that for every finite graph $\G=(V,E)$
with degree at most $\Delta$ and every  $v\in V, t\in \Z_+$:
\begin{align}\label{eq:UpperLowerIS}
|\log\Phi_{\G}(v,t)-\log\Psi_{\G}(v,t)|\le C\rho^t.
\end{align}
As a result, $\G$ satisfies exponential SSM for  $\lambda$ satisfying (\ref{eq:lambdaBound}).
\end{theorem}
\begin{proof} The details of the proof can be found in \cite{weitzCounting}. It is shown there
that the absolute value in (\ref{eq:UpperLowerIS}) is upper bounded by the same quantity, when applied
to $\G=\T_{\Delta,t}$ - the $\Delta$-regular depth-$t$ tree. Then a classical results by Spitzer~\cite{Spitzer75}
and Kelly~\cite{KellyHardCore} are invoked to show the existence of $C$ and $\rho$. The existence of $\rho$ then
implies SSM via Proposition~\ref{prop:hardcorereduction} and observing that changing values of $\sigma$
for nodes $u$ which have distance bigger than $t$ from a given node $v$, does not affect the values
of $\Phi_{\G}(v,t)$ and $\Psi_{\G}(v,t)$.
\end{proof}

\subsection{Free energy and surface pressure on $\Z^d$. Numerical results}
We are now equipped to obtain bounds  on the free energy and the surface pressure for the hard-core
model on $\Z^d$.

Denote by $\Phi(t)$ and $\Psi(t)$ the values of $\Phi_{\G}(v,t),\Psi_{\G}(v,t)$ when applied
to a graph $\G=\Z^d_{\prec 0}\cap B_n$, for sufficiently large $n$ and $v=0$. Observe that
the values $\Phi_{\G}(v,t),\Psi_{\G}(v,t)$ are the same for all values of $n$ sufficiently larger than $t$
(for example $n\ge t+1$ suffices). Thus the notations are well-defined.
The following relations are the basis for computing bounds on the free energy and surface pressure.
\begin{coro}\label{coro:UpperLowerISLattice}
For every $t\in\Z_+$ and $\lambda$ satisfying (\ref{eq:lambdaBound})
\begin{align}
-\log \Phi(2t)&\le \mathcal{P}(d,\lambda)\le -\log \Psi(2t) \label{eq:UpperLowerISLattice1}\\
-\log \Psi(2t+1)&\le \mathcal{P}(d,\lambda)\le -\log \Phi(2t+1) \label{eq:UpperLowerISLattice2}.
\end{align}
\end{coro}

\begin{proof}
By Theorem~\ref{theorem:CorrelationDecayIS} we have exponential SSM.
Thus
\begin{align*}
\pr_{\Z^d_{\prec 0}}(0\notin I)=\lim_{n\rightarrow\infty}\pr_{B_n\cap \Z^d_{\prec 0}}(0\notin I).
\end{align*}
$B_n\cap \Z^d_{\prec 0}$ is finite graph for which bounds from Lemma~\ref{lemma:UpperLowerIS} are applicable.
\end{proof}

Our algorithm for computing $\mathcal{P}(d,\lambda)$ and $s\mathcal{P}(\Z^d,\lambda)$
is based on relations (\ref{eq:UpperLowerISLattice1})
and (\ref{eq:UpperLowerISLattice2}), and will be called \emph{Sequential Cavity Algorithm} or shortly SCA.

We have numerically computed values $\Phi(t),\Psi(t)$ for the cases $d=2,3,4$ using the chessboard pattern method.
Our results provide the following bounds on the free energy for the case $\lambda=1$. Since previous
bounds were stated in terms of $\exp(\mathcal{P}(d,1)$, we  do the same here:
\begin{align*}
&1.503034\le \exp(\mathcal{P}(2,1))\le 1.503058 \\
&1.434493\le \exp(\mathcal{P}(3,1))\le 1.449698 \\
&1.417583\le \exp(\mathcal{P}(4,1))\le 1.444713
\end{align*}
The computations were done at the level $t=27$ for the case $d=2$, $t=16$ for the case $d=3$ and $t=12$
for the case $d=4$.
Our lower bound for the case $d=2$ is weaker than the previous best known $1.503047782$~\cite{CalkinWilf},
which is already very close to the presumably correct but unproven value stated in~\cite{Baxter99}.
However, our upper bound is stronger than the previous best known $1.5035148$~\cite{CalkinWilf}.
We are not aware of any estimates for the case $d=3$. Thus we believe our bounds are the best known.
We have not done computations of the surface pressure and we are not aware of any previously existing benchmarks.
Note that in the case  $d=3$ we have $\Delta=6$, and $\lambda=1$ no longer satisfies (\ref{eq:lambdaBound}).
Thus we have no guarantee that SCA will provide converging estimates as $t$ increases.
The correctness of our bounds for this case is guaranteed by Corollary~\ref{coro:MainResultUpperLower}.
It is encouraging to
see that the bounds are close and based on this fact we conjecture that $\lambda=1$ corresponds to the uniqueness
regime.

It is instructive to compare our numerical results, which were obtained using the chess-pattern approach
(identity (\ref{eq:PressureIS1chesspattern})) with
results which could be obtained directly from (\ref{eq:PressureIS1}). The computations based on (\ref{eq:PressureIS1})
for the case $d=2,\lambda=1$ at depth $t=12$ lead to bounds
$1.0942\le \exp(\mathcal{P}(2,1))\le 1.8377$. At the same time, the computations using chess pattern method
at the depth only $t=3$ already lead to a much tighter bounds
$1.4169\le \exp(\mathcal{P}(2,1))\le 1.5565$.

Notice, that while Theorem~\ref{theorem:CorrelationDecayIS} is not used in computing actual bounds on the
free energy and surface pressure, it provides the guarantee for the quality of such bounds. Let us
use it now to analyze the computation effort required to obtain a particular level of accuracy in bounds.

\begin{prop}\label{prop:ComputationalComplexityIS}
For every $d,\lambda<(d-1)^{d-1}/(d-2)^d$ and  $\epsilon>0$  SCA produces an $\epsilon$-additive
estimate of $\mathcal{P}(d,\lambda)$ and $s\mathcal{P}(\Z^d,\lambda)$ in time
$({1\over \epsilon})^{O(1)}$, where
the constant in $O(\cdot)$ depends on $\lambda$ and $d$.
\end{prop}
\ignore{This estimate has a deceivingly good depends on $d$, especially compared to
numerical effort $\exp(O((1/\epsilon)^{d-1}))$ of the transfer matrix mathod.
In fact it is important to keep in mind that
our method works only for $\lambda$ in the range (\ref{eq:lambdaBound}) which pushes $\lambda$ down
and affects $\rho$, as $d$ increases.
The transfer matrix method, on the other hand, does not require an upper bound on $\lambda$.

We see that for fixed
$\lambda$ and $d$ we have an implied fixed $\rho$ and then, treating $d$ and $\rho$ as constants,
SCA has performance $({1\over \epsilon})^{O(\cdot)}$ as opposed to $\exp(O((1/\epsilon)^{d-1}))$ of the
transfer matrix method.
}

\begin{proof}
Applying Theorem~\ref{theorem:CorrelationDecayIS}, an additive error $\epsilon$ is achieved
provided that $C\rho^t<\epsilon$ or \\
$t\ge \log(C/\epsilon)\log^{-1}(1/\rho)=O(\log(1/\epsilon))$.
Hence the result for free energy
follows from Lemma~\ref{lemma:ComputationalComplexityISGeneral}.

For surface pressure observe that applying Theorem~\ref{theorem:CorrelationDecayIS}
\begin{align*}
\Big|\log \pr_{\Z^d_{\prec 0}}(\sigma_{0}=0)-\log \pr_{\Z^d_{j,k,+,\prec 0}}(\sigma_{0}=0)\Big|\le C\rho^k.
\end{align*}
A similar bound holds for $\pr_{\Z^d_{j,-k,-,\prec 0}}(\sigma_{0}=0)$. Thus if we take $k_0$ such that
$C\rho^{k_0}/(1-\rho)<\epsilon$, then the partial sum in (\ref{eq:SurfacePressureGeneral}) corresponding
to terms $k\ge k_0$ is at most $2\epsilon$. The required $k_0$ is
$O\big(\log((1-\rho)C/\epsilon)/\log(1/\rho)\big)=O(\log(1/\epsilon))$.
For the remaining terms $k<k_0$ we compute $\pr_{\Z^d_{j,k,+,\prec 0}}(\sigma_{0}=0)$ and
$\pr_{\Z^d_{j,-k,-,\prec 0}}(\sigma_{0}=0)$ using SCA with accuracy $\hat\epsilon=\epsilon/k_0$.
Since $\log(1/\hat\epsilon)=\log(1/\epsilon)+\log\log(1/\epsilon)=O(\log(1/\epsilon))$, the
result then follows from our estimate for computing  $\mathcal{P}(d,\lambda)$.
\end{proof}

\section{Monomer-dimer (matching) model}\label{section:matchings}
The monomer-dimer model is defined by spin values $S=\{0,1\}$ assigned to edges of a graph $G=(V,E)$.
A set of edges $M\subset E$ is a matching if no two edges in $M$ are incident.
Sometimes term partial matching is used to contrast with full matching, which is a matching with size $|V|/2$
(namely every node is incident to an edge in the matching).
The edges of $M$ are called dimers
and nodes in $V$ which are not incident to any edge in $M$ are called monomers.
We set $H(0,0)=H(0,1)=H(1,0)=h(0)=0,H(1,1)=\infty, h(1)=\beta$. The Gibbs measure is defined via (\ref{eq:GibbsEdges}).
As in the case of hard-core model, it is convenient to introduce
$\lambda=\exp(-\beta)>0$.  Similarly to the hard-core model we have
for finite graphs $\G$
\begin{align}
Z_{\G}=\sum_{M}\lambda^{|M|} \label{eq:PartitionFunctionM}
\end{align}
were the summation is over all matchings in $\G$. The summation in
(\ref{eq:PartitionFunctionM}) is called a matching polynomial in the combinatorics literature.
We denote by $\mbM$ a random matching chosen according to the Gibbs measure, when it is unique. In the case of finite
graphs
\begin{align*}
\pr_{\G}(\mbM=M)=Z_{\G}^{-1}\lambda^{|M|}.
\end{align*}

The monomer-dimer model is a close relative of the hard-core model, even though its properties are substantially different.
For example this model does not exhibit a phase transition and is always in the uniqueness
regime~\cite{HeilmanLieb}. Moreover, it satisfies
the SSM property for all activities $\lambda$ as we shall shortly see. Corresponding analogues of
Proposition~\ref{prop:hardcorereduction} will be stated later once exponential SSM is asserted.

\subsection{Recursion, sequential cavity algorithm and correlation decay}
We now establish an analogue of (\ref{eq:RecursionIndSet}) for the monomer-dimer model.
The proof of this result can be found in~\cite{BayatiGamarnikKatzNairTetali} and is omitted. It is similar
to the proof of (\ref{eq:RecursionIndSet}).
In the following, with a slight abuse of notation we write $v\in M$ if matching $M$ contains an edge
incident to $v$.
\begin{theorem}\cite{BayatiGamarnikKatzNairTetali}\label{theorem:RecursionMatchings}
For every  finite graph $\G=(V,E)$ and $v\in V$
\begin{align}\label{eq:RecursionMatching}
\pr_{\G}(v\notin \mbM)={1\over 1+\lambda\sum_{u\in N_{\G}(v)}\pr_{\G_0}(u\notin \mbM)}
\end{align}
where  $\G_0$ is induced by $V\setminus \{v\}$ and $\sum_{u\in N_{\G}(v)}=0$ when $v$ is an isolated node.
\end{theorem}
The further development in this subsection mirrors the one of Subsection~\ref{subsection:AlgorithmIS},
yet the conclusion will be different - the monomer-dimer model exhibits SSM for \emph{all} values
of $\lambda$. This will lead to an algorithm for computing the free energy and surface pressure
for monomer-dimer model for every value $\lambda>0$.

Given a finite graph $\G$, for every subgraph $\hat\G=(\hat V,\hat E)$ of $\G$, every node $v\in
\hat\G$ and every  $t\in \mathbb{Z}_+$ we introduce a quantity
$\Phi_{\hat\G}(v,t)$ defined inductively as follows. In the context of monomer-dimer model
this quantity stands for (approximate) probability that $v\notin \mbM$ in the subgraph $\hat\G$.
\begin{align}\label{eq:PhirecursionM}
\Phi_{\hat \G}&(v,t)
=\left\{
                      \begin{array}{ll}
                        1, & \hbox{$t=0$ or $N(v)=\emptyset$;} \\
                        (1+\lambda\sum_{1\le i\le k}\Phi_{\G_0}(v_i,t-1))^{-1}, &
                    \hbox{$t>0$ and  $N(v)=\{v_1,\ldots,v_k\}\ne\emptyset$.}
                    \end{array}
                    \right.
\end{align}
Here $\hat\G_{0}$ is induced by $\hat V\setminus \{v\}$.
If $\Phi_{\G_0}(v_i,t-1)=\pr_{\G_0}(v_i\notin \mbM)$ for all $i$
then $\Phi_{\G}(v,t)=\pr_{\G}(v\notin \mbM)$.

Similarly,  introduce  $\Psi_{\hat G}(v,t)$ with the only exception that
$\Psi_{\hat \G}(v,0)=0$ for all $\hat\G$ and $v$. The following proposition follows
from Theorem~\ref{theorem:RecursionMatchings} and the definitions of $\Phi$ and $\Psi$
using a simple induction argument.
\begin{lemma}\label{prop:UpperLowerM}
For every $v\in V, t\in\Z_+$
\begin{align*}
\Psi_{\G}(v,2t)&\le \pr_{\G}(v\notin \mbM)\le \Phi_{\G}(v,2t) \\
\Phi_{\G}(v,2t+1)&\le \pr_{\G}(v\notin \mbM)\le \Psi_{\G}(v,2t+1).
\end{align*}
\end{lemma}
Next we provide bounds on the computational effort required to compute $\Phi$ and $\Psi$.

\begin{lemma}\label{lemma:ComputationalComplexityMGeneral}
For every finite graph $\G$ with degree $\Delta$, $v\in\G$ and $t$, the values
$\Phi_{\G}(v,t),\Psi_{\G}(v,t)$ can be computed in time $\exp(O(t\log \Delta))$, where
the constant in $O(\cdot)$ is universal.
\end{lemma}

\begin{proof}
The result follows immediately from the recursive definitions of $\Phi$ and $\Psi$.
\end{proof}
The correlation decay property is formulated in the following proposition which
is proved in~\cite{BayatiGamarnikKatzNairTetali}. Let
\begin{align}\label{eq:rho}
\rho=\Big(1 -\frac{2}{\sqrt{1+\lambda \Delta} +1}\Big)^{1/2}
\end{align}

\begin{theorem}[\cite{BayatiGamarnikKatzNairTetali}]\label{theorem:CorrelationDecayM}
For every  $\Delta\ge 2,\lambda>0$,
for every graph $\G$ with degree at most $\Delta$ and every node $v$
\begin{align}\label{eq:UpperLowerM}
|\log\Phi_{\G}(v,t)-\log\Psi_{\G}(v,t)|\le \rho^t\log(1+\lambda\Delta).
\end{align}
As a consequence, every graph $\G$ satisfies exponential SSM for all $\lambda>0$.
\end{theorem}
We now state the analogues of Proposition~\ref{prop:hardcorereduction} and Corollary~\ref{corollary:IS}.
The proofs are very similar and omitted.
Given any set $W\subset E$, let $N(W)=W\cup\{e:\exists e'\in W, e\sim e'\}$.
\begin{prop}\label{prop:matchingsreduction}
Given a graph $\G=(V,E)$,
for every mutually exclusive sets $A,W_1,W_2\subset E$ and spin assignment $(s_e), e\in A$ on $A$
the following identity holds
with respect to the unique Gibbs measure.
\begin{align}\label{eq:reductionM}
\pr_{\G}(1\{e\in \mbM\}&=s_e, \forall e\in A|W_1\cap \mbM=\emptyset,W_2\subset \mbM)=\notag\\
&=\pr_{\hat\G}(1\{e\in \mbM\}=s_e, \forall e\in A),
\end{align}
where $\hat\G$ is the subgraph obtained from $\G$ by removing edges $W_1\cup N(W_2)$.
\end{prop}

We use notations $\mathcal{P}(d,\lambda)$ and $s\mathcal{P}(d,\lambda)$ for the free energy and the surface pressure
for the monomer-dimer model on $\Z^d$ as well. As a corollary of Theorem~\ref{theorem:MainResultGeneralEdges},
and Proposition~\ref{prop:matchingsreduction} we obtain
\begin{coro}\label{corollary:M}
For every $\lambda>0$ and $d$
\begin{align}\label{eq:PressureM1}
\mathcal{P}(d,\lambda)=-\log \pr_{\Z^d_{\prec 0}}(0\notin\mbM)=
-{1\over 2}\log \pr_{\Z^d_{\prec 0,\even}}(0\notin\mbM)
\end{align}
\ignore{
Also for every $(a_i)\in \R_{>0}^d$
\begin{align}\label{eq:SurfacePressureIS1}
s\mathcal{P}(d,\lambda)&=
\sum_{1\le j\le d}a_j^{-1}\sum_{k=0}^\infty\Big(2\log \pr_{\Z^d_{\prec 0}}(0\notin\mbM)-
\log \pr_{\Z^d_{j,k,+,\prec 0}}(0\notin\mbM)- \\
&-\log \pr_{\Z^d_{j,-k,-,\prec 0}}(0\notin\mbM)\Big) \notag.
\end{align}
}
\end{coro}

\subsection{Free energy and surface pressure on $\Z^d$. Numerical results}
We now obtain bounds  on the free energy and surface pressure for the monomer-dimer
model on $\Z^d$.
Again denote by $\Phi(t)$ and $\Psi(t)$ the values of $\Phi_{\G}(v,t),\Psi_{\G}(v,t)$ when applied
to any graph $\G=\Z^d_{\prec 0}\cap B_n$ in the monomer-dimer context, for sufficiently large $n$ and $v=0$.
The relations (\ref{eq:UpperLowerISLattice1}) and (\ref{eq:UpperLowerISLattice2}) hold as well and the proof
is very similar.

Our algorithm for computing $\mathcal{P}(d,\lambda)$ and $s\mathcal{P}(\Z^d,\lambda)$
is again based on computing $\Phi(t)$ and $\Psi(t)$ and is again called
\emph{Sequential Cavity Algorithm} (SCA).
We now report numerical results on computing $\mathcal{P}(d,\lambda)$.
We have computed values $\Phi(t),\Psi(t)$  for  $d=2,3,4$ and a range of values $\lambda$
using the chessboard pattern method. Our upper and lower bounds for free energy are presented in
Table~\ref{table:pressureMatchings}.
The depth levels $t=14,9,7$ were used for the cases $d=2,3,4$. As expected, our bounds are high quality
for lower $\lambda$ and then degrade as $\lambda\rightarrow\infty$. Each computation run took about
3 minutes on a workstation and we have not made an attempt to obtain very accurate bounds for each value $\lambda$.
However for the case of interest $\lambda=1$ we ran our algorithm for larger depths.
For the case $d=2$ a very accurate rigorous estimate $0.6627989727\pm 0.0000000001$
is due to Friedland and Peled~\cite{FriedlandPeled},
shown non-rigorously earlier by Baxter~\cite{Baxter68}. We have not made an attempt to improve this bound.
However for the case $d=3$ we can significantly improve the best known bound
$0.7850\le  \mathcal{P}(3,1) \le 0.7863$ due to Friedland
et al~\cite{FriedlandPeled},\cite{FriedlandGurvits},\cite{FriedlandKropLundowMarkstrom}. At depth
 $t=19$ we obtained
estimates $0.78595\le \mathcal{P}(3,1) \le 0.78599$, which is two orders of magnitude improvement.
One should also note that the lower bounds $0.7845$ in~\cite{FriedlandGurvits} and $0.7850$ in~\cite{FriedlandKropLundowMarkstrom}
were obtained using Friedland-Tveberg inequality which
provides a bound for general regular graphs. Thus, while highly accurate for the case $\Z^3$, this bound
is not improvable by running some numerical procedure longer or on a faster machine.
The previous best known \emph{numerical}
estimate $\mathcal{P}(3,1)\ge 0.7653$~\cite{FriedlandPeled}, is based on the transfer matrix method is weaker.
We have also obtained bounds for $d=4$ for which no prior computations are available. We obtained
$0.8797\le  \mathcal{P}(4,1) \le 0.8812$. The computations were done at depth $t=14$.

\ignore{
\begin{table}
\begin{center}
\begin{tabular}{|c|c|c|}
\hline
$\lambda$ & $d=2$ lower & $d=2$ upper \\
  \hline
  \hline
    0.1000 & 2.3219  &  2.3219 \\
    0.2000 & 1.6802 & 1.6802 \\
    0.3000 & 1.3451  &  1.3451\\
    0.4000 & 1.1361  &  1.1361 \\
    0.5000 & 0.9934  &  0.9934\\
    0.6000 & 0.8902 & 0.8902\\
    0.7000 & 0.8122 & 0.8122\\
    0.8000 & 0.7513  &  0.7513\\
    0.9000 & 0.7026  &  0.7026\\
    1.0000 & 0.6628  &  0.6628\\
    1.5000 & 0.5389  &  0.5392\\
    2.0000 & 0.4742 & 0.4758\\
    3.0000 & 0.4056  &  0.4139\\
    4.0000 & 0.3665  &  0.3855\\
    5.0000 & 0.3390 & 0.3704\\
    6.0000 & 0.3174  &  0.3614\\
    7.0000 & 0.2993  &  0.3558\\
    8.0000 & 0.2837  &  0.3519\\
    9.0000 & 0.2697  &  0.3493\\
   10.0000 & 0.2570 & 0.3473\\
   15.0000 & 0.2064  &  0.3426\\
   20.0000 & 0.1685  &  0.3409\\
   30.0000 & 0.1151  &  0.3397\\
   40.0000 & 0.0811  &  0.3393\\
   50.0000 & 0.0592 & 0.3391\\
 \hline
\end{tabular}
\caption{}
\end{center}
\end{table}
}

\begin{table}
\begin{center}
\begin{tabular}{|c|c|c|c|c|c|c|}
\hline
$\lambda$ & $d=2$ L & $d=2$ U & $d=3$ L & $d=3$ U & $d=4$ L & $d=4$ U \\
  \hline
  \hline
    0.1000 & 2.3219 & 2.3219 & 2.3311 & 2.3311 & 2.3399 & 2.3399 \\
    0.2000 & 1.6802 & 1.6802 & 1.7096 & 1.7096 & 1.7363 & 1.7363 \\
    0.3000 & 1.3451 & 1.3451 & 1.3959 & 1.3959 & 1.4395 & 1.4395 \\
    0.4000 & 1.1361 & 1.1361 & 1.2050 & 1.2050 & 1.2621 & 1.2624 \\
    0.5000 & 0.9934 & 0.9934 & 1.0770 & 1.0770 & 1.1442 & 1.1453 \\
    0.6000 & 0.8902 & 0.8902 & 0.9853 & 0.9855 & 1.0599 & 1.0629 \\
    0.7000 & 0.8122 & 0.8122 & 0.9164 & 0.9171 & 0.9961 & 1.0025 \\
    0.8000 & 0.7513 & 0.7513 & 0.8627 & 0.8642 & 0.9457 & 0.9568 \\
    0.9000 & 0.7026 & 0.7026 & 0.8196 & 0.8224 & 0.9044 & 0.9215 \\
    1.0000 & 0.6628 & 0.6628 & 0.7840 & 0.7887 & 0.8695 & 0.8937 \\
    1.5000 & 0.5389 & 0.5392 & 0.6677 & 0.6890 & 0.7475 & 0.8163 \\
    2.0000 & 0.4742 & 0.4758 & 0.5982 & 0.6436 & 0.6671 & 0.7840 \\
    3.0000 & 0.4056 & 0.4139 & 0.5079 & 0.6055 & 0.5564 & 0.7586 \\
    4.0000 & 0.3665 & 0.3855 & 0.4455 & 0.5908 & 0.4793 & 0.7492 \\
    5.0000 & 0.3390 & 0.3704 & 0.3972 & 0.5836 & 0.4208 & 0.7448 \\
    6.0000 & 0.3174 & 0.3614 & 0.3575 & 0.5797 & 0.3740 & 0.7423 \\
    7.0000 & 0.2993 & 0.3558 & 0.3239 & 0.5772 & 0.3354 & 0.7408 \\
    8.0000 & 0.2837 & 0.3519 & 0.2947 & 0.5757 & 0.3026 & 0.7399 \\
    9.0000 & 0.2697 & 0.3493 & 0.2691 & 0.5746 & 0.2743 & 0.7392 \\
   10.0000 & 0.2570 & 0.3473 & 0.2465 & 0.5738 & 0.2496 & 0.7387 \\
   15.0000 & 0.2064 & 0.3426 & 0.1643 & 0.5719 & 0.1621 & 0.7376 \\
   20.0000 & 0.1685 & 0.3409 & 0.1147 & 0.5713 & 0.1110 & 0.7372 \\
   30.0000 & 0.1151 & 0.3397 & 0.0627 & 0.5708 & 0.0592 & 0.7369 \\
   40.0000 & 0.0811 & 0.3393 & 0.0386 & 0.5706 & 0.0359 & 0.7368 \\
   50.0000 & 0.0592 & 0.3391 & 0.0259 & 0.5705 & 0.0239 & 0.7368 \\
   \hline
\end{tabular}
\caption{Upper (U) and lower (L) bounds on free energy for $d=2,3,4$}\label{table:pressureMatchings}
\end{center}
\end{table}

The following proposition gives a bound on the numerical complexity of SCA.
\begin{prop}\label{prop:ComputationalComplexityM}
For every $d\ge 2,\lambda>0, \epsilon>0$   SCA produces an $\epsilon$-additive
estimate of $\mathcal{P}(d,\lambda)$ and $s\mathcal{P}(d,\lambda,a)$ in time
$\log(1+2\lambda d)(1/\epsilon)^{O\big((\lambda d)^{1\over 2}\log d\big)}$,
where the constant in $O(\cdot)$ is universal.
\end{prop}

Again for constant $\lambda,d$ we obtain performance $(1/\epsilon)^{O(1)}$, which is
a qualitative improvement over the numerical effort $\exp(O((1/\epsilon)^{d-1}))$ of the
transfer matrix method.

\begin{proof}
We have $\Delta=2d$. Applying Theorem~\ref{theorem:CorrelationDecayM}, an additive error $\epsilon$ is achieved
provided that $\rho^t\log(1+2\lambda d)<\epsilon$ or
\begin{align*}
t\ge (-\log \rho)^{-1}\Big(\log{1\over \epsilon}+\log\log(1+2\lambda d)\Big).
\end{align*}
Applying (\ref{eq:rho}) we have $\log\rho=O(\frac{1}{\sqrt{1+2\lambda d} +1})
=O(\frac{1}{\sqrt{\lambda d} })$. The result for $\mathcal{P}(d,\lambda)$ then follows from this estimate and
Lemma~\ref{lemma:ComputationalComplexityMGeneral}. The result for $s\mathcal{P}(d,\lambda)$ is established
using the same line of reasoning as for Proposition~\ref{prop:ComputationalComplexityIS}.
\end{proof}


\section{Conclusions}\label{section:conclusions}
Several statistical mechanics models besides hard-core and monomer-dimer models
fit our framework, yet were not discussed in this paper. One such model
is Ising model and its generalization, Potts (coloring) model~\cite{SimonLatticeGases}.
The Ising model satisfies our Assumption~\ref{assumption:s*}
by making $s^*=1$ or $-1$. Thus the representation Theorem~\ref{theorem:MainResultGeneral} holds,
provided (exponential) SSM holds.
However, in order to turn it into a useful method for obtaining provably converging bounds, we need an analogue
of Theorems~\ref{theorem:CorrelationDecayIS} and \ref{theorem:CorrelationDecayM}, namely the correlation decay
property on a computation tree. Such a result is indeed established in \cite{GamarnikKatz}, but
for an Ising model with very weak interactions. Thus it is of interest to strengthen the
result in~\cite{GamarnikKatz} and obtain some concrete estimates for the Ising model on $\Z^d$. We are not aware
of earlier benchmark results for this model. The situation with the Potts model is similar with the exception
of hard-core Potts model (proper coloring). In this case the required $s^*$ does not exist, as for every
color $s, H(s,s)=\infty$. However, our method can be extended by considering for example a periodic coloring
of nodes $v\succ 0$ or using the chess-pattern version Theorem~\ref{theorem:MainResultGeneralChessPattern} of
our main result. The required correlation decay result is established in~\cite{GamarnikKatz} for the
case $q>2.86\Delta$ and $q,\Delta$ appropriately large constants. Thus in order to turn this result
into a useful method for computing free energy and surface pressure for coloring
model on $\Z^d$, one needs to deal with these constants
explicitly. Also a convenient monotonicity present in the hard-core and monomer-dimer models is lost
in the Potts model case, which makes  application of Corollary~\ref{coro:MainResultUpperLower} harder.

\section*{Acknowledgements}
The first author wishes to thanks I. Sinai for an inspiring conversation, R. Pemantle and Fa. Y. Wu for providing
several relevant references, S. Friedland for updating us on the state of the art numerical estimates,
J. Propp and Domino server participants for responding to some
of our questions, and F. Martinelli and C. Borgs for enlightening discussions on Strong Spatial Mixing.

\bibliographystyle{amsalpha}

\newcommand{\etalchar}[1]{$^{#1}$}
\providecommand{\bysame}{\leavevmode\hbox to3em{\hrulefill}\thinspace}
\providecommand{\MR}{\relax\ifhmode\unskip\space\fi MR }
\providecommand{\MRhref}[2]{%
  \href{http://www.ams.org/mathscinet-getitem?mr=#1}{#2}
}
\providecommand{\href}[2]{#2}

\end{document}